\documentclass[11pt]{article}      

\usepackage{amssymb}      
\usepackage{amsmath}      
\usepackage{amsthm}      
\usepackage{amsbsy}      
\usepackage{amscd}      
\usepackage[dvips]{graphicx}

\theoremstyle{plain}      
\newtheorem{theorem}{Theorem}[section]      
\newtheorem{lemma}{Lemma}[section]      
\newtheorem{corollary}{Corollary}[section]
\newtheorem{proposition}{Proposition}[section]      
\newtheorem{conjecture}{Conjecture}[section]      
      
\newtheorem{definition}{Definition}[section]          
\theoremstyle{remark}      
\newtheorem{remark}{Remark}[section]

\newtheorem{example}{Example}[section]

\newcommand{\Q}{{\mathbb{Q}}}        
\newcommand{\Z}{{\mathbb{Z}}}   
   
\newcommand{\C}{{\mathbb{C}}}      
\newcommand{\R}{{\mathbb{R}}}      

\newcommand{\ro}{{\widetilde{\rho}}}

\begin{document}

\date{\today}

\title{Images of quantum representations of mapping class groups and Dupont-Guichardet-Wigner quasi-homomorphisms}         
\author{\begin{tabular}{cc}      Louis Funar &  Wolfgang Pitsch\\      
\small \em Institut Fourier BP 74, UMR 5582       &\small \em Departament de Matem\`atiques \\      
\small \em University of Grenoble I &\small \em Universitat Aut\`onoma de Barcelona    \\      
\small \em 38402 Saint-Martin-d'H\`eres cedex, France      
&\small \em 08193 Bellaterra (Cerdanyola del Vall\`es), Espana  \\      
\small \em e-mail: {\tt louis.funar@ujf-grenoble.fr}      
& \small \em e-mail: {\tt pitsch@mat.uab.es} \\      
\end{tabular}      
}


\maketitle

\begin{abstract}
We prove that either 
the images of the mapping class groups by quantum representations are not 
isomorphic to higher rank lattices or else the kernels have a large 
number of normal generators. Further we show that the images of the mapping class groups have nontrivial
2-cohomology, at least for small levels.  For this purpose we considered a series of quasi-homomorphisms 
on mapping class groups extending previous work of Barge and Ghys 
 \cite{BG} and of Gambaudo and Ghys  \cite{GG}. 
These quasi-homomorphisms are pull-backs of the Dupont-Guichardet-Wigner 
quasi-homomorphisms on pseudo-unitary groups along quantum 
representations.

\vspace{0.1cm}



\end{abstract}

\section{Introduction and statements}

The main motivation of this paper is to obtain new  information about the 
images of mapping class groups by quantum representations by analyzing their 2-cohomology.  
McMullen (\cite{McM}) addressed the question of the arithmeticity of Burau representations of 
braid groups at roots of unity and Venkataramana (\cite{Venka,Venka2}) solved it affirmatively in the case where 
the order of the root is bounded by twice the number of strands. Burau representations are 
particular examples of quantum representations in genus zero. Whether the image of quantum representations of 
mapping class groups of higher genus is arithmetic or thin seems to be a challenging problem with possible implications for the fine structure of mapping class groups. One additional difficulty  in both the present case and the Burau representation at roots of unity of higher order 
is the absence of unipotents.  Another, seemingly unrelated, question which arose recently is the determination of the kernel 
of the quantum representations at a fixed level, to be compared with the 
normal subgroup generated by given powers of Dehn twists.  It is known that the intersection 
of infinitely many  of these kernels is trivial, according to the asymptotic  faithfulness result by Andersen (\cite{And}, see also 
\cite{FWW,MN} for different proofs). 
Our aim is to prove first that the two questions above are directly related, in particular  
arithmeticity implies a large number of normal generators for the kernel, thus many other besides   
powers of Dehn twists. 
Our second result shows that in infinitely many cases the real 2-cohomology of the image of the quantum representations is nontrivial and hence these images  are not virtually free.

One ingredient in this work is the relation between Burau and quantum representations,  
which we use  to estimate the signature of Hermitian forms invariant by the mapping class groups.   
As a consequence, quantum representations are Zariski dense within semi-simple groups with a large number 
of pseudo-unitary factors (see also \cite{F2}).  We  then apply  Matsushima's vanishing theorem to prove 
that either the images of quantum representations are not higher 
rank irreducible lattices, or else the number of normal generators 
of the kernels of the quantum representations are bounded from below 
by linear functions on the level of the representation.
In a second part we consider the family of quasi-homomorphisms 
on mapping class groups defined in \cite{F2},  
extending and inspired by previous work of Barge and Ghys  \cite{BG} and of 
Gambaudo and Ghys  \cite{GG}.  
These  quasi-homomorphisms are constructed as trivializations of
pull-backs of Dupont-Guichardet-Wigner cocycles 
along quantum  representations of mapping class groups 
$M_g$ of oriented surfaces of genus $g \geq 2$  
into pseudo-unitary groups.  Although 
Bestvina and Fujiwara proved in \cite{BF}  that there are uncountably many 
quasi-homomorphisms on mapping class groups, which could be derived 
using the action of mapping class groups on curve complexes, it seems 
that there are very few explicit ones. Explicit computations using arithmetic properties of 
the signatures from the first part give then the non-triviality 
of 2-cohomology classes on the image of the quantum representations, at least 
for small levels.

\subsection{Quantum representations}

In \cite{BHMV}, Blanchet, Habegger, Masbaum and Vogel  defined the TQFT functor $\mathcal V_{p}$, for every 
integer $p\geq 3$  
and a primitive root of unity $\zeta$ of order $2p$. These TQFT should correspond to the so-called 
$SU(2)$-TQFT, for even $p$ and to the $SO(3)$-TQFT, for odd $p$ (see also \cite{LW} for another version of 
$SO(3)$-TQFT). It is known that these TQFT determine and are determined by a series of projective representations of  
the mapping class groups.

\begin{definition}\label{qrep}
Let $p\in\Z_+$, $p\geq 5$ and $\zeta$ be a primitive 
$2p$-th root of unity. 
\begin{enumerate}
\item The  quantum representation $\rho_{p,\zeta}$ 
is the projective representation of  the mapping class group 
associated to $\mathcal V_{p}$, the TQFT  at the root of unity  $\zeta$. 
\item We denote 
by $\ro_{p,\zeta}$ the linear representation of the central extension 
$\widetilde{M_g}$ of the mapping class groups $M_g$ of the genus $g$ closed oriented surface which resolves the 
projective ambiguity of $\rho_{p,\zeta}$ (see \cite{Ger,MR}).  
\item Furthermore, $N(g,p)$ denotes the dimension of the space of conformal 
blocks associated by the TQFT $\mathcal V_{p}$ to the closed 
oriented surface of genus $g$. 
\end{enumerate}
\end{definition}

Recall now  that $M_g$ is perfect when $g\geq 3$ and 
that the universal central extension $\widetilde{M_g}^{\rm u}$ 
of $\mathcal M_g$ is a subgroup of index 12 in 
the central extension $\widetilde{M_g}$  (see \cite{MR}). We will often consider 
the restriction of $\ro_{p,\zeta}$ to the perfect subgroup  $\widetilde{M_g}^{\rm u}$ since the later has no other central extensions than itself.

\begin{remark}
The TQFT $\mathcal V_p$ is unitary in the case  
$\zeta=A_p$, where 
\[ A_p=\left\{\begin{array}{ll}
-\exp\left(\frac{2\pi i}{2p}\right), & {\rm if}\: p\equiv 0({\rm mod}\: 2);\\
(-1)^{\frac{p-1}{2}}\exp\left(\frac{(p+1)\pi i}{2p}\right) , & {\rm if}\: p\equiv 1({\rm mod}\: 2).\\
\end{array}\right. \]
Notice a slight change with respect to  the convention 
\cite{F2} where a typo arose in the expression for 
odd $p$. 
\end{remark}

For prime $p\geq 5$  we denote 
by ${\mathcal O}_p$ the  ring of cyclotomic integers 
${\mathcal O}_p=\Z[\zeta_p]$, if   
$p\equiv -1({\rm mod}\: 4)$ and ${\mathcal O}_p=\Z[\zeta_{4p}]$, if  
$p\equiv 1({\rm mod}\:4)$ respectively, where $\zeta_r$ denotes a primitive
$r$-th root of unity.
The main result of \cite{GM} states that there exists a free ${\mathcal O}_p$\,-lattice 
$S_{g,p}$ in the $\C$-vector space of conformal 
blocks associated by  the TQFT ${\mathcal V}_p$ to the genus $g$
closed orientable surface and a non-degenerate Hermitian  
${\mathcal O}_p$-valued form on 
$S_{g,p}$ both invariant under  the action of $\widetilde{M_g}$ via the representation 
$\widetilde{\rho}_{p,\zeta}$. 
Therefore the image of the mapping class group consists of 
unitary matrices (with respect to the Hermitian form) with 
entries in ${\mathcal O}_p$. Let  $\mathbb U_{p,g}({\mathcal O}_p)$ and  $P\mathbb U_{p,g}({\mathcal O}_p)$ be the group of all such matrices and respectively its quotient by scalars. 

When $p$ is  prime $p\geq 5$ and $g\geq 3$, then it is known that 
$\ro_{p,A_p}$ takes values in $S\mathbb U_{p,g}$ (see \cite{DW,FP}).  
It is known that $S\mathbb U_{p,g}({\mathcal O}_p)$ is an irreducible 
lattice in a semi-simple algebraic group $\mathbb G_{p,g}$ obtained 
by the so-called restriction of scalars construction from the 
totally real cyclotomic field $\Q(\zeta_p+\overline{\zeta_p})$ to $\Q$. 
Specifically, the group  $\mathbb G_{p,g}$ is a product 
$\prod_{\sigma\in S(p)}S\mathbb U_{p,g}^{\sigma}$. Here $S(p)$ stands for 
a set of representatives of the classes of complex 
embeddings $\sigma$ of $\mathcal O_p$ 
modulo complex conjugacy. 
The factor $S\mathbb U_{p,g}^{\sigma}$ is the 
special unitary group associated to the 
Hermitian form conjugated by $\sigma$, thus corresponding to some 
Galois conjugate root of unity.  

Denote  
by $\ro_p$ and $\rho_p$ the representations  
$\prod_{\sigma\in S(p)} \ro_{p,\sigma(A_p)}$ and 
$\prod_{\sigma\in S(p)} \rho_{p,\sigma(A_p)}$, respectively. 
Notice that the real Lie group $\mathbb G_{p,g}$ 
is a semi-simple algebraic group defined over $\Q$.

In \cite{F2} the first author proved that 
$\ro_p(\widetilde{M_g})$ is a discrete Zariski dense subgroup 
of $\mathbb G_{p,g}(\R)$ whose projections onto the simple factors of 
$\mathbb G_{p,g}(\R)$ are topologically dense, for $g \geq 3$ and  $p\geq 5$ 
prime, $p\equiv -1({\rm mod }\; 4)$. 

\begin{remark}
\begin{enumerate}
\item Notice that, when 
$p\equiv 1 ({\rm mod }\; 4)$ the image of 
$\ro_p(\widetilde{M_g})$ is contained in 
$\mathbb G_{p,g}(\Z[i])$ and thus it is a discrete 
Zariski dense subgroup of $\mathbb G_{p,g}(\C)$.  Thus we have to replace 
each factor $SU(m,n)$ of $\mathbb G_{p,g}(\R)$ by its complexification 
$SL(m+n,\C)$. There are a number of essential changes 
to be made if we wish to extend Theorem \ref{alter} to this case, 
contrary to the situation in \cite{F2}. However for Theorem \ref{lower} 
the discreteness is not an issue.    
\item When $p=2r$, for a prime $r\geq 5$,  
according to (\cite{BHMV}, Theorem 1.5) there is an isomorphism 
of TQFTs between $\mathcal V_{2r}$ and $\mathcal V_2'\otimes \mathcal V_r$. 
Furthermore the image of the TQFT representation associated to $\mathcal V_2'$ 
is finite. Thus,  the restriction of $\ro_p$ 
to the finite index subgroup  $\ker \ro_2'\subset \widetilde{M_g}$ 
is the tensor product of a trivial representation and $\ro_r$, hence is a direct sum of 
copies of $\ro_r$. The projection on a factor gives us 
a homomorphism $\pi:\rho_{2r}(\ker \ro_2')\to \mathbb G_{r,g}$. 
Therefore, up to passing to a finite index subgroup of $\widetilde{M_g}$ 
the image $\pi\circ \ro_{2r}$ is a discrete Zariski dense subgroup of $\mathbb G_{r,g}$.
\end{enumerate}
\end{remark}

\subsection{Main results}

The questions addressed here concern the description of the image of $\rho_{p,\zeta}$ 
and its kernel. The first problem is whether the image of $\rho_{p,\zeta}$ is of finite index 
in $P\mathbb U_{p,g}(\mathcal O_p)$, and in particular a higher rank lattice. 
Let $M_g[p]$ denote the (normal) subgroup of $M_g$ generated by the 
$p$-th powers of all  Dehn twists. It is known that $M_g[p]\subset \ker \rho_{p,\zeta}$ 
and the second problem is whether this inclusion is strict.  
This was stated in \cite{Mas1} and in unpublished notes by J\o{}rgen Andersen. 
For instance this inclusion is an equality  when the surface is a one-holed torus and the representations are 
2-dimensional (see \cite{FK1,Mas1}) or a 4-holed sphere (see \cite{AMU}).
Notice that $M_g[p]$ has a small normally generating system.

Our first result states that whenever 
$\ro_{p}(\widetilde{M_g}^{\rm u})$ is isomorphic to a higher rank lattice 
the  group  $\ro_{p}(\widetilde{M_g}^{\rm u})$ should be the 
quotient of $\widetilde{M_g}^{\rm u}$ by a {\em large} number 
of relations, growing linearly with $p$.

To state this properly we need more notation. 
Set $s_{p,g}$ for the number of simple non-compact factors of the 
semi-simple Lie group $\mathbb G_{p,g}(\R)$. 
We also write $s_{p,g}^{\ast}$ for the number 
of such factors of non-zero signature i.e. of the form 
$SU(m,n)$ with $m\neq n$, $mn\neq 0$. Each simple factor is associated to 
a primitive root of unity $\zeta$ of order $2p$ having positive 
imaginary part. 
Those $\zeta$ corresponding to non-compact simple factors or non-compact
with non-zero signature will be called {\em non-compact roots} 
and respectively {\em non-compact roots of non-zero signature}.   
Denote also by $r_{p,g}$ the minimal number (possibly infinite) of normal generators 
of $\ker \ro_p$ 
within $\widetilde{M_g}^{\rm u}$, 
namely the minimum number of relators to be added  in order 
to obtain the quotient $\ro_p(\widetilde{M_g}^{\rm u})$.

\begin{theorem}\label{alter}
Let $g\geq 4$, $p$ prime, $p\equiv -1({\rm mod } \;4)$. 
Either $\ro_{p}(\widetilde{M_g}^{\rm u})$  is not isomorphic to a 
higher rank lattice, or else $r_{p,g} \geq s_{p,g}$. Moreover, 
\[s_{p,g}\geq \left\lceil\frac{g-3}{2(g-1)}p +\frac32\right\rceil, \;\; {\rm for }\; p\geq 2g-1, \; g\geq 4,\] 
where $\lceil x\rceil$ denotes the smallest integer greater or equal to $x$. 
\end{theorem}

A consequence of our theorem above is the following:

\begin{corollary}\label{notlattice}
Let $g\geq 4$, $p$ prime, $p\equiv -1({\rm mod } \;4)$ such that 
$p \geq 2g-1$. Then the quotient 
$M_g/M_g[p]$  is not isomorphic to a 
higher rank lattice. 
\end{corollary}

The way one proves this theorem is by finding an upper bound for  
the dimension of the cohomology group $H^2(\ro_p(\widetilde{M_g}^{\rm u}),\R)$ 
in terms of the number of normal generators. This is carried on in section \ref{proofalter}. 
The necessary estimates for $s_{p,g}$ and the real rank of $\mathbb G_{p,g}$ are provided in sections \ref{estimates} and \ref{rrank}, after having set the notation for the skein TQFT in section \ref{tqft}. 
 
Lower bounds for these dimensions are more difficult to obtain and this is the subject of 
the second part of the article. Here we use the aforementioned family of quasi-homomorphisms 
on mapping class groups  arising as trivializations of pull-backs of Dupont-Guichardet-Wigner cocycles 
along quantum  representations.  We  first need an explicit formula for these quasi-homomorphisms, 
which will be stated in  Proposition \ref{quasi}  section \ref{dgw}. 
Then computations of signatures arising in non-unitary TQFTs obtained in section \ref{compute}  for small values of the level   
provide the necessary ingredients for the following result:

\begin{theorem}\label{lower}
For $p\in\{5,7,9\}$ and infinitely many values of $g$  we have 
$\dim H^2(\ro_p(\widetilde{M_g}^{\rm u}),\R)\geq 1$. 
\end{theorem}

Since $\rho_p(M_g)$ is of finite index within 
$\ro_p(\widetilde{M_g}^{\rm u})$, from the 5-term exact sequence 
in cohomology it follows that: 
\begin{corollary}
For $p\in\{5,7,9\}$ and infinitely many values of $g$  we have 
$\dim H^2(\rho_p(M_g),\R)\geq 1$. 
\end{corollary}

An immediate consequence is the fact that $\rho_p(M_g)$ is {\em not} a virtually free 
group. This can be improved, as follows. 
For a group $\Gamma$ which is virtually torsion-free we denote by ${\rm vcd}(\Gamma)$ its virtual cohomological dimension, i.e. 
the cohomological dimension of any of its finite index torsion-free subgroups (see \cite{Brown}, VIII.11).

\begin{proposition}\label{vcd}
If $p\not\in\{2,3,4,6,8,12\}$ and $g\geq 2$, $(p,g)\neq (10,2)$ then we have:
\[ {\rm vcd}(\widetilde\rho_p(\widetilde\Gamma_g))\geq g + \left[\frac{g-2}{2}\right].\]
In particular, $\widetilde\rho_p(\widetilde\Gamma_g)$ is not virtually a free product of finite groups. 
\end{proposition}

Moreover the cohomology classes in Theorem \ref{lower}  are not related to known classes on mapping class groups:

\begin{proposition}\label{prop h2rhoiszero}
For any $g \geq 2$, the map induced in  cohomology in degree $2$   $$\rho_p^\ast : H^2(\rho_p(M_g),\R) \rightarrow H^2(M_g,\R)$$ is the  trivial (zero) map. 
\end{proposition}

\begin{remark}
The restriction to $p\in \{5,7,9\}$ comes from our inability 
to obtain  modular properties for the 
signatures of TQFTs for general $p$. A general theory for these 
is beyond the scope of this paper and partial results in 
this direction will appear in \cite{CF}. 
We expect the result to hold for all primes $p$. 
However these cases with small $p$ are already 
interesting since the representations $\ro_p$ are known to be 
Zariski dense in the corresponding semi-simple Lie groups $\mathbb G_{p,g}$. 
Our method could improve this lower bound for specific values of $p$ and $g$, but  
couldn't do better than $\left[\frac g2\right]+1$ without 
additional information about the group $\ro_p(\widetilde{M_g}^{\rm u})$. 
The arithmetic progressions 
above are rather explicit, for instance $g\equiv 1 ({\rm mod }\; 24)$ 
is convenient for $p\in\{5,7\}$. 
\end{remark}

{\bf Acknowledgements}. 
We are indebted to C. Blanchet, M. Brandenbursky,  
F. Costantino, F. Dahmani, K. Fujiwara, E. Ghys,  
V. Guirardel, G. Kuperberg, G. Masbaum, G. McShane, M. Sapir,  
and A. Zuk for useful discussions and advice and to the referee for  
several suggestions which considerably improved the presentation of this paper. 
The first author was partially supported by 
the ANR 2011 BS 01 020 01 ModGroup and the second author was supported 
by the FEDER/MEC grants MTM2010-20692 and MTM2013-42293.

\section{Quasi-homomorphisms on mapping class group quotients}

\subsection{Restriction homomorphisms and proof of Theorem \ref{alter}}\label{proofalter}

\begin{proposition}\label{hom}
We have $\dim H^2(\ro_p(\widetilde{M_g}^{\rm u}),\R)\leq r_{p,g}$, if $g\geq 3$.
\end{proposition}
\begin{proof}
The  5-term exact sequence in cohomology associated to the exact sequence  
\[ 1\to \ker \ro_p \to \widetilde{M_g}^{\rm u} \to \ro_p(\widetilde{M_g}^{\rm u})\to 1,\]
gives us:
\[ 0=H^1(\widetilde{M_g}^{\rm u},\R)\to {\rm Hom}(\ker \ro_p, \R)^{\widetilde{M_g}^{\rm u}}\stackrel{\iota}{\to} H^2(\ro_p(\widetilde{M_g}^{\rm u}),\R)\to 
H^2(\widetilde{M_g}^{\rm u},\R)=0.\]

By exactness of the sequence above $\iota$ is an isomorphism and hence 
identifies ${\rm Hom}(\ker \ro_p, \R)^{\widetilde{M_g}^{\rm u} }$ with 
$H^2(\ro_p(\widetilde{M_g}^{\rm u}),\R)$. 
The next lemma shows that $\dim{\rm Hom}(\ker \ro_p, \R)^{\widetilde{M_g}^{\rm u} }\leq r_{p,g}$ and Proposition \ref{hom} follows. 
\end{proof}

\begin{lemma}
Assume that $r_{p,g}$ is finite and  let $\{a_1,a_2,\ldots,a_{r_{p,g}}\}$ be a minimal system of 
normal generators for $\ker\ro_p$ within $\widetilde{M_g}^{\rm u}$. 
Then the evaluation 
homomorphism $E:{\rm Hom}(\ker \ro_p, \R)^{\widetilde{M_g}^{\rm u} }
\to \R^{r_{p,g}}$, given by 
$ E(f)= (f(a_1),f(a_2),\ldots, f(a_n))$
is injective. 
\end{lemma}
\begin{proof}
Any element $x\in \ker\ro_p$ is a product 
$x=\prod_i g_ia_ig_i^{-1}$, for some $g_i\in \widetilde{M_g}^{\rm u}$.
Since $f\in {\rm Hom}(\ker \ro_p, \R)^{\widetilde{M_g}^{\rm u} }$ is conjugacy 
invariant we have 
$ f(x)=\sum_{i}f(g_ia_ig_i^{-1})=\sum_if(a_i)$ and the Lemma follows.
\end{proof}

\begin{proposition}\label{bounded}
If $s_{p,g} > r_{p,g}$ then 
$\ro_{p}(\widetilde{M_g}^{\rm u})$ is not a lattice in $\mathbb G_{p,g}$.
\end{proposition}
\begin{proof}
Recall from \cite{F2} that $\mathbb G_{p,g}$ is a real  
semi-simple linear algebraic semi-simple group defined over $\Q$. 
Since $\mathbb G_{p,g}$ is obtained by restriction of scalars from an  anisotropic 
unitary group it follows that all elements of $\mathbb G_{p,g}(\Z)$ are 
semi-simple, as being obtained as Galois conjugates of  unitary 
and hence diagonalizable matrices. Therefore, by  Borel's Theorem,
$\mathbb G_{p,g}(\Z)$ is a cocompact lattice in $\mathbb G_{p,g}(\R)$. This was also 
noticed in \cite{MRe}.

We know as part of  Matsushima's vanishing theorem 
that for cocompact lattices $\Gamma$ in semi-simple Lie groups 
$\mathbb G$ the restriction homomorphism 
$H^j(\mathbb G,\R)\to H^j(\Gamma, \R)$ is an isomorphism 
as long as $j\leq {\rm rk}_{\R}\mathbb G - 1$ (see \cite{BW}, ch. 7, Prop. 4.3).
We will show below in  Proposition \ref{rank}, section \ref{rrank} 
that $\mathbb G_{p,g}(\R)$ is of rank at least $3$ for any odd $p\geq 5$, 
and hence $H^2(\mathbb G_{p,g}(\R),\R)\to H^2(\Gamma,\R)$ is an isomorphism for any lattice $\Gamma$ in 
$\mathbb G_{p,g}(\R)$.

Now, $\mathbb G_{p,g}(\R)$ is a product of $s_{p,g}$ 
pseudo-unitary groups of type $SU(m,n)$, each factor being a simple 
group of isometries of some irreducible Hermitian space.  
Then by \cite{GW} we have that $H^2(\mathbb G_{p,g}(\R),\R)=\R^{s_{p,g}}$ is the 
vector space generated 
by the set of Dupont-Guichardet-Wigner classes of the simple factors. 
In particular, if $s_{p,g} >  r_{p,g}$ then the restriction map 
$H^2(\mathbb G_{p,g}(\R),\R)\to H^2(\ro_{p}(\widetilde{M_g}^{\rm u}),\R)$
cannot be an isomorphism by dimensional reasons and so 
$\ro_{p}(\widetilde{M_g}^{\rm u})$ can not be isomorphic to a lattice in $\mathbb G_{p,g}(\R)$.   
\end{proof}

\begin{proof}[Proof of Theorem \ref{alter}]
 Assume that $\ro_{p}(\widetilde{M_g}^{\rm u})$  is  isomorphic to a higher rank irreducible lattice. 
 For  $p$ as in the hypothesis one knows that $\ro_{p}(\widetilde{M_g}^{\rm u})$ is a discrete subgroup of $\mathbb G_{p,g}(\R)$.  Then, by Margulis super-rigidity theorem (see \cite{Mar}) and the arithmeticity of 
lattices in higher rank Lie groups there exists a finite index subgroup  of $\ro_{p}(\widetilde{M_g}^{\rm u})$ which is a 
lattice in a product $P$ of simple factors of $\mathbb G_{p,g}(\R)$.  Therefore the Zariski closure 
of $\ro_{p}(\widetilde{M_g}^{\rm u})$ is contained in the subgroup $P$. 
On the other hand, as observed before, $\ro_{p}(\widetilde{M_g}^{\rm u})$  is Zariski dense in 
 $\mathbb G_{p,g}$ and hence $P=\mathbb G_{p,g}(\R)$, so that 
 $\ro_{p}(\widetilde{M_g}^{\rm u})$ must be a lattice in $\mathbb G_{p,g}(\R)$. 
 Now Proposition \ref{bounded} settles the first part of the Theorem. 
 We postpone the proof of the lower bound $\
 s_{p,g}\geq \left\lceil\frac{g-3}{2(g-1)}p +\frac32\right\rceil$, for $p\geq 2g+1$,  
 until the next section \ref{estimates}, see  Proposition \ref{simplefactors}.
\end{proof}

\begin{proof}[Proof of Corollary \ref{notlattice}]
First,  $M_g[p]$ is normally generated by 
the $p$-th powers of Dehn twists 
along a set of curves containing one simple closed curve  for each integer $1\leq h\leq \frac g2$ which is 
bounding a sub-surface of genus $h$ along with one non-separating simple curve. This gives an upper bound of 
$t_g= 1+\left[\frac g2\right]$  for the number of normal generators  of $M_g[p]$, which is   
independent on $p$.

Assume that $M_g/M_g[p]$ is a higher rank lattice $\Gamma$ in the semi-simple 
Lie group $H$. We know that there exists a surjection of $\Gamma$ onto 
$\rho_p(M_g)$ which is a discrete Zariski dense subgroup of $P\mathbb G_{p,g}$. 
By Margulis super-rigidity theorem (see \cite{Mar}) there exists a surjective 
continuous homomorphism $H\to P\mathbb G_{p,g}(\R)$ covering this surjection. 
Therefore the number of virtual Hermitian simple non-compact 
factors of $H$ is at least the number $s_{p,g}$  associated to 
$P\mathbb G_{p,g}(\R)$.

The proof of Proposition  \ref{bounded}  applied to the surjection 
$M_g\to M_g/M_g[p]$ shows that  
\[ \dim H^2(M_g/M_g[p],\R)\leq t_g.\] 
Finally,  by Matsushima's vanishing theorem we also have 
$\dim H^2(\Gamma,\R) \geq s_{p,g}$. This leads to a contradiction 
for $p$ large enough, as stated. 
\end{proof}

\section{Estimates concerning the TQFT Hermitian form}
\subsection{The setting of the skein TQFT}\label{tqft}
We briefly review  the properties of the TQFT $\mathcal V_p$ and refer to 
\cite{BHMV} for more details.  A TQFT   is a 
functor from the category of surfaces into the category of finite dimensional vector spaces. 
Specifically, the objects of the first category are closed oriented surfaces endowed with 
colored banded points and morphisms between two objects are cobordisms 
decorated by uni-trivalent ribbon graphs compatible with the banded points.
A banded point on a surface is a point with a tangent vector at that point, or equivalently 
a germ of an oriented interval embedded in the surface. There is a corresponding  
surface with colored boundary obtained by deleting a small neighborhood of the 
banded points and letting the boundary circles inherit the colors of the respective points.

The vector space  associated by the functor $\mathcal V_p$ to a surface 
is called the {\em space of conformal blocks}.  Let $\Sigma_g$ denote the genus $g$ closed orientable surface, 
$H_g$ be a genus $g$ handlebody with $\partial H_g=\Sigma_g$. Assume given a finite set  $\mathcal Y$ of banded 
points on $\Sigma_g$. Let $G$ be a uni-trivalent ribbon graph embedded in $H_g$ in such a way that 
$H_g$ retracts onto $G$, its univalent vertices are the banded points $\mathcal Y$ and it has no other intersections 
with $\Sigma_g$. 

For  an odd number $p\geq 5$, called the {\em  level} of the TQFT,  we consider the 
{\em set of colors} in level $p$ to be $\{0,2,4,\ldots,p-3\}$.  
An edge coloring of $G$ is called {\em $p$-admissible} if the triangle inequality is 
satisfied at any trivalent vertex of $G$ and the sum of the three colors around a 
vertex is bounded by $2(p-2)$. There is a similar description of $p$-admissibility for even $p$. 

Fix a coloring of the banded points $\mathcal Y$. Then there exists a basis of the space of conformal blocks associated to the surface $(\Sigma_g, \mathcal Y)$ with the  colored banded points (or the corresponding surface with colored boundary) 
which is indexed by the set of all $p$-admissible colorings of $G$ 
extending the boundary coloring. We  denote by $W_{g}$ the vector space associated to the closed surface 
$\Sigma_g$ without banded points, or equivalently, where all banded points are given the color $0$. 

In fact an admissible $p$-coloring of $G$ provides an element of the skein module 
$S_{\zeta}(H_g)$ of the handlebody evaluated at a primitive $2p$-th root of unity $\zeta$. This skein element is obtained by cabling 
the edges of $G$ by the Jones-Wenzl idempotents prescribed by the coloring. 
Let $\overline{H}_g$ denote the complementary handlebody in the 3-sphere $S^3$. Then there is a 
sesquilinear form: 
\[ \langle \;,\; \rangle: S_{\zeta}(H_g)\times S_{\zeta}(\overline{H}_g)\to \C\]
defined by 
\[ \langle x, y \rangle= \langle x \sqcup y \rangle.\]
Here $x\sqcup y$ is the element of $S_{\zeta}(S^3)$ obtained by the disjoint union of  $x$ and $y$ in 
$H_g\cup\overline{H}_g=S^3$, 
and $\langle \; \rangle: S_{\zeta}(S^3)\to \C$ is the Kauffman bracket invariant. 

Eventually the space of conformal blocks $W_g$ is the quotient $S_{\zeta}/\ker \langle\;,\; \rangle$ by the left kernel of the sesquilinear form above. It follows that $W_g$ is endowed with an induced {\em Hermitian form} $H_{\zeta}$. 
The projections of skein elements associated to the $p$-admissible colorings of a trivalent graph $G$ as above form an orthogonal basis of $W_g$ with respect to $H_{\zeta}$. 

Let $G'\subset G$ be a uni-trivalent subgraph  whose degree one vertices are colored, corresponding to a 
sub-surface $\Sigma'$ of $\Sigma_g$ with colored boundary. The projections in $W_g$ of skein elements associated to 
the $p$-admissible colorings of $G'$ form an orthogonal basis of the space of conformal blocks 
associated to the surface $\Sigma'$ with colored boundary components. 

There is a geometric action of the mapping class groups of the handlebodies $H_g$ and $\overline{H}_g$ respectively
on their skein modules and hence on the space of conformal blocks. Moreover, these actions extend to the projective  
action $\rho_{p,\zeta}$ of $M_g$ on $W_g$ respecting the Hermitian form $H_{\zeta}$. Notice that the 
mapping class group of an essential (i.e. without annuli or disks complements) 
sub-surface $\Sigma'\subset \Sigma_g$ is a subgroup of $M_g$ which preserves the subspace of conformal blocs 
associated to $\Sigma'$ with colored boundary.  This kind of restriction to sub-surfaces is an essential ingredient 
in the next section.

The functor $\mathcal V_p$ associates to a handlebody $H_g$ the projection of the skein element 
corresponding to the trivial coloring of the trivalent graph $G$ by $0$. The invariant associated 
to a closed 3-manifold is given by pairing the two vectors associated to handlebodies in a Heegaard decomposition 
of some genus $g$ and taking into account the twisting by the gluing mapping class action on $W_g$. 

One should notice that the skein TQFT $\mathcal V_p$ is unitary, in the sense that $H_{\zeta}$ is 
a positive definite Hermitian form when $\zeta=A_p$, as chosen in the introduction. 
The main concern of the present article is the case of a general primitive $2p$-th root of unity, in which case 
the isometries of $H_{\zeta}$ form a pseudo-unitary group.

\subsection{Estimations on $s_{p,g}$}\label{estimates}
\begin{proposition}\label{simplefactors}
If $g\geq 4$ and $p\equiv -1 ({\rm mod }\; 4)$, then 
\[s_{p,g}\geq \left\lceil\frac{g-3}{2(g-1)}p +\frac32\right\rceil, \;\; {\rm for }\; p\geq 2g+1,\] 
where $\lceil x\rceil$ denotes the smallest integer greater or equal to $x$. 
\end{proposition}
\begin{proof}
This statement is essentially combinatorial, 
as the Hermitian form  $H_{\zeta}$ on the space of conformal blocks  
is given  rather explicitly in \cite{BHMV} in its diagonal form. 
Nevertheless the combinatorial-arithmetic problem of counting the 
roots of unity for which the entries of $H_{\zeta}$ are all positive seems 
rather complicated. We propose here an alternative way to bound from below  
$s_{p,g}$ by restricting the problem from mapping class 
groups to braids, where computations are immediate. 
Although not sharp our estimates are linear in $p$.

There is an obvious injection of  the pure braid 
group $PB_{g-1}$ on $(g-1)$ strands into $M_{g}$, when $g\geq 3$. 
Specifically, if the $g$-holed sphere is embedded in $\Sigma_g$, in 
such a way that its complement consists of $g$ one-holed tori, then the 
map induced at the level of their mapping class groups is injective. 
Now, the pure mapping class group of the $(g-1)$-holed disk is an extension of $PB_{g-1}$ by the 
free abelian group $\Z^{g-1}$ of Dehn twists along $(g-1)$ boundary components.  This extension 
splits non-canonically, thus providing an embedding of $PB_{g-1}$ into $M_{g}$.

The restriction of the representation $\rho_{p,\zeta}$ of $M_{g}$  
to $PB_{g-1}$ is not irreducible.  Set $W_{0,g}$ for the space of conformal blocks associated 
to the disk with $(g-1)$ holes, whose boundary circles are labeled by the 
colors $(2g-4,2,2,\ldots,2)$, the first label corresponding to the disk boundary. 
In order to admit an extension to a $p$-admissible coloring we need to impose the condition $p\geq 2g-1$. 
Then the restriction  $\rho_{p,\zeta}|_{PB_{g-1}}$ leaves invariant the subspace $W_{0,g}\subset W_g$. 
Moreover this representation naturally extends to one of the full braid group $B_{g-1}$, since the colors 
of $(g-1)$ boundary circles of the sub-surface coincide. Eventually the projective representation of $B_{g-1}$ lifts 
to a linear representation of $B_{g-1}$.  Indeed, central extensions by $\Z$ of the braid groups $B_{g-1}$ 
are trivial, as Arnold (\cite{Arnold}) proved that $H^2(B_n,\Z)=0$. 
We will still denote this linear lift by  $\rho_{p,\zeta}|_{B_{g-1}}$. 

Recall now that the (reduced) Burau representation $\beta_{k}:B_k\to GL(k-1,\Z[q,q^{-1}])$, for $k\geq 3$,  
is defined on the standard generators $g_1,g_2,\ldots,g_{k-1}$ of the braid group $B_k$ on $k$-strands by the formulas:  
\[ \beta_q(g_1)=\left(\begin{array}{cc}
-q & 0 \\
-1  & 1 \\
\end{array}
\right) \oplus {\mathbf 1}_{k-3},\]
\[ \beta_q(g_j)={\mathbf 1}_{j-2}\oplus 
\left(\begin{array}{ccc}
1 & -q & 0 \\
0 & -q & 0 \\
0 & -1  & 1 \\
\end{array}
\right) \oplus {\mathbf 1}_{k-j-2}, \:\: {\rm for} \:\: 2\leq j\leq k-2,\]
\[ \beta_q(g_{k-1})={\mathbf 1}_{k-3}\oplus 
\left(\begin{array}{cc}
1 & -q \\
0 & -q \\ 
\end{array}
\right). 
\]
By taking $q\in \C^*$ we obtain a representation $\beta_{k}(q)$ with values in $GL(k-1,\C)$. 
The representations $\beta_k(q)$ are irreducible unless $q$ is  a nontrivial $k$-th root of unity, in which case it 
has a $(k-2)$-dimensional irreducible summand denoted $\hat{\beta}_k(q)$.  
Following Formanek (see \cite{For}) we call a complex representation of $B_k$ of {\em Burau-type} if 
it is isomorphic to the tensor product of $\beta_{k}(q)$ (or $\hat{\beta}_k(q)$) with some 1-dimensional representation. 
The later are all of the form $\chi(y)$, where $\chi(y)(g_j)=y\in \C^*$, for $j\leq k-1$.

\begin{lemma}\label{Burau}
The representation $\rho_{p,\zeta}|_{B_{g-1}}$ on $W_{0,g}$ is of Burau-type. 
\end{lemma}
\begin{proof}
This is known to be true for $g=4,5$ (see e.g. \cite{F, FK2}). 
By induction on $g$ one shows  that $\dim W_{0,g}=g-2$. 
Explicit computations as those in \cite{F} (for even $p$) show that  the elements $\rho_{p,\zeta}|_{B_{g-1}}(g_i)$ have only two nontrivial eigenvalues. Up to rescaling the images of $g_i$ (i.e. twisting by a one-dimensional representation) 
the two nontrivial eigenvalues are  $1$ and $-\zeta^8$. Also the image of $g_i$ is a pseudo-reflection. 

Formanek  proved in (\cite{For}, Theorem 10, 22) that 
irreducible representations of $B_{g-1}$  of dimension at most $g-2$ are either 1-dimensional or of Burau-type, 
if $g \geq 8$ or $g\geq 6$ and the image of $g_i$ is a pseudo-reflection.  
When $p\geq 2g-1$ is prime, $\hat{\beta}_k(q)$ (with $q^{g-1}=1$) cannot be a summand of $\rho_{p,\zeta}|_{B_{g-1}}$, because one eigenvalue of $g_i$ is not a $(g-1)$-th root of unity. 

Notice that $\rho_{p,\zeta}$, and hence $\rho_{p,\zeta}|_{B_{g-1}}$, is semi-simple because 
it is Galois conjugate to the unitary representation $\rho_{p,A_p}$.

If $\rho_{p,\zeta}|_{B_{g-1}}$  were not irreducible then it would split as a direct sum of 1-dimensional representations. 
This is a contradiction, as the restriction of $\rho_{p,\zeta}|_{B_{g-1}}$  to $B_3\subset B_{g-1}$ is of  Burau-type.

Therefore, up to twisting by some $\chi(y)$ (the explicit value of $y$ is not needed) 
$\rho_{p,\zeta}|_{B_{g-1}}$ is equivalent to the Burau representation $\beta_{q_p}$ of $B_{g-1}$ at the root of unity $q_p$, 
where $q_p$ is given by $q_p=\zeta^{8}$, for odd $p$. 
\end{proof}

\vspace{0.2cm}\noindent
Now, for $k\geq 3$ the Burau representation  $\beta_k(q)$ of $B_k$ 
has an invariant Hermitian form defined by Squier in \cite{Squier}. 
Squier's original Hermitian form is degenerate when $q$ is a root of unity of order 
$n\leq k$. A slightly modified version $H^S_{q}$ of this form can be found in \cite{McM,We}, where it is shown that 
it is non-degenerate unless $q$ is a $k$-th root of unity. 
Since a Burau-type representation is irreducible it admits a unique invariant Hermitian form, up to a 
real scalar. 

In particular, the restriction of $H_{\zeta}$ to the space of conformal blocks is a real multiple of $H^S_{\zeta^8}$. 
This also follows from stronger results from \cite{FLW,Ku} concerning the density of images of 
Burau-type representation.

The signature of the form $H^S_{q}$ is given in (\cite{McM}, Corollary. 3.2).  
Squier's form is definite (either positive or negative) 
if and only if ${\rm arg}(q_p)\in \left(-\frac{2\pi}{g-1},\frac{2\pi}{g-1}\right)$ 
(see also \cite{Abdul}, Lemma 9).
If we set $\zeta=\exp\left(\frac{(2k+1)\pi i}{p}\right)$,  
then it suffices to restrict to those integral $k\in\{0,1,\ldots,\frac{p-3}{2}\}$. This condition on 
${\rm arg }(q_p)$ amounts to counting all such integers $k$ for which in addition  
\[ 2s\pi-\frac{2\pi}{g-1}\leq \frac{4(2k+1)\pi}{p}\leq 2s\pi \frac{2\pi}{g-1},\;\; {\rm where }\;\;  s\in\{0,1,2,3,4\}.\]
The number of such integers 
is at most $\frac{p}{g-1}-\frac{3}{2}$.

In \cite{F2} the first author proved that the factors of $\mathbb G_{p,g}$, for 
$p\equiv -1({\rm mod }\; 4)$ are in one-one correspondence with 
the $\frac{p-1}{2}$ primitive $2p$-th roots of 
unity up to conjugacy. If we discard the compact ones we derive  
that the  number of non-compact factors in $\mathbb G_{p,g}$ is  at least 
$\frac{g-3}{2(g-1)}p +\frac32$.

\end{proof}

\begin{remark}
It seems that there are precisely two conjugate values for which $H_{\zeta}$ is 
positive when $p\geq 5$ is odd prime and four values 
(obtained by conjugacy or changing the sign) when $p$ is twice an odd prime, 
respectively, unless $H_{\zeta}$ is totally positive. 
A similar statement might hold for all (not necessarily 
prime) odd  large enough $p$.  
\end{remark}

\begin{remark}
Similar estimates hold true for $g\in\{2,3\}$, 
by using the homomorphisms 
$PB_3\to M_2$  and $PB_4\to M_3$ from \cite{FK2}. 
We skip the details.   
\end{remark}

\subsection{Estimates for the rank of $\mathbb G_{p,g}(\R)$}\label{rrank}
\begin{proposition}\label{rank}
For $g\geq 2$, prime $p\geq 7$ and $p\equiv -1({\rm mod }\; 4)$ the  
real rank of $\mathbb G_{p,g}(\R)$  is  at least $2$.
Furthermore, for $g\geq 4$ and odd $p\geq 5$  each simple  
non-compact factor of $\mathbb G_{p,g}(\R)$  has rank  at least $2$.
Moreover, the real rank of $\mathbb G_{p,g}(\R)$ is at least 
$\left(\left\lceil\frac{g-3}{2(g-1)}p +\frac32\right\rceil\right)\left(\frac{p-1}{2}\right)^{g-3}$, 
for $g\geq 4$, $p\geq 2g+1$ and $p\equiv -1({\rm mod }\; 4)$.    
\end{proposition}
\begin{proof}
Let $W_g^{\pm}(\zeta)$ be a maximal positive/negative 
subspace of the space $W_g$ 
of conformal blocks in genus $g$, for the Hermitian form $H_{\zeta}$. 
Consider  a  separating  curve $\gamma$  on 
the closed orientable surface $\Sigma_g$ whose complementary sub-surfaces 
have genus $g-1$ and $1$ respectively. If we label $\gamma$ by $0$ then 
the spaces of conformal blocks associated to these two sub-surfaces 
are isometrically identified with the spaces of conformal blocks 
of the closed surfaces obtained by capping off the boundary components.  
Therefore we have natural isometric embeddings  
$W_{g-1} \otimes W_1 \hookrightarrow W_g$.  
It is well-known that $W_1=W_1^+(\zeta)$ is positive for any $\zeta$. 
Therefore we obtain the following isometric embeddings: 
$ W_{g-1}^+(\zeta)\otimes W_1\hookrightarrow W_g^+(\zeta)$ and 
$ W_{g-1}^-(\zeta)\otimes W_1\hookrightarrow W_g^-(\zeta)$. 
In particular, we have for odd $p$
\[ \dim W_g^+(\zeta) \geq (\dim W_1)^g=\left(\frac{p-1}{2}\right)^{g}.\]
\begin{lemma}
If  $\zeta$ is such that $W_3^+(\zeta)=W_3$, 
then $W_g^+(\zeta)=W_g$, i.e. the simple factor associated 
to $\zeta$ is compact. 
\end{lemma} 
\begin{proof}
For a $p$-admissible coloring $X$  of the trivalent graph $G$ with $g$ loops we denote by the same letter $X$ 
the corresponding vector of the basis of $W_g$ defined in section \ref{tqft} above.  
For a vertex $v$ we denote by $a_v,b_v,c_v$ the 
colors of the three edges incident to $v$ and for any edge $e$ 
we denote by $c_e$ the color of the edge $e$, as prescribed by $X$. 
The Hermitian norm of such a vector $X$ was computed in (\cite{BHMV}, 4.11), as follows: 
\[ H_{\zeta}(X,X) =\eta^{g-1} \prod_{v\in V(G)} \langle a_v,b_v,c_v\rangle \cdot 
\prod_{e\in E(G)} \langle c_e\rangle^{-1},\]
where $\eta$ is a constant independent of the genus, $V(G)$ denotes the set of 
vertices  and $E(G)$ the set of edges of the graph $G$. 
The precise values of the symbols $\eta$, $\langle a,b,c\rangle$ 
and $\langle a\rangle$ in terms of quantum numbers 
are given in \cite{BHMV} but  they will not  be explicitly  
needed in the sequel. We only need to know that all of them are 
real numbers.

Observe also that the positivity of the Hermitian form in genus $3$ implies 
the positivity for genus $2$, as well. 
Now, there are two graphs with two loops and without leaves (degree one vertices), 
the theta graph and the graph made of two loops joined by a segment. 
The above formula for a vector corresponding to a coloring of the 
theta graph shows that: 
\[\eta \langle a \rangle \langle b \rangle\langle c \rangle > 0,\]
for any $p$-admissible triple $a,b,c$ at a vertex. Therefore 
all symbols $\langle a\rangle$ have the same sign as $\eta$. 
Using the other graph with two loops we find that 
\[ \langle a, a, b \rangle  \langle c, c, b \rangle > 0,\]
for every $p$-admissible coloring for which the symbols above are defined. 
Thus the sign of $\langle a, a, b \rangle $ is $\epsilon_b\in\{-1,+1\}$ 
and it only depends on $b$. 
Consider next a graph made of three loops joined together by means of 
a tree with one vertex and three edges, each edge having its endpoint on 
one loop. Take an arbitrary $p$-admissible triple of colors $a,b,c$ for the three edges of 
the tree and color the loops in a $p$-admissible way. This is always possible, no matter how we chose 
the $p$-admissible triple $a,b,c$. The formula above implies that: 
\[  \langle a, b, c \rangle \epsilon_a\epsilon_b\epsilon_c > 0.\]
But now it is immediate that for any vector $X$ corresponding to a 
colored trivalent graph   without leaves with $g\geq 2$ loops we have 
$H_{\zeta}(X,X) >0$. This implies that the Hermitian 
form on every space of conformal blocks associated to a 
closed orientable surface is positive definite. 
\end{proof}

It follows that either $W_g^+(\zeta)=W_g$ is positive or else 
\[ \dim W_g^-(\zeta)\geq (\dim W_1)^{g-3}\dim W_3^-(\zeta)\geq 
\left(\frac{p-1}{2}\right)^{g-3}.\]
The two formulas above show that the rank of each simple non-compact factor 
of $\mathbb G_{p,g}$ is at least $\left(\frac{p-1}{2}\right)^{g-3}$.

\vspace{0.2cm}\noindent
On the other hand if $p$ is odd  and $(p,g)\neq (2,5)$ then, by direct calculation 
one obtains that the Hermitian form associated to the 
1-holed torus with the boundary circle colored by $2$ 
is not totally positive. 
The argument above implies that the real rank of $\mathbb G_{p,g}(\R)$ is at least 
$2$. A similar statement is valid for even $p\geq 14$. 
\end{proof}

\begin{remark}
When $g=2$ and $p=7$  the group $\mathbb G_{p,g}(\R)$ is the product of two    
pseudo-unitary groups $SU(11,3)\times SU(10,4)$. 
When $g=3$ and $p=7$ the group  $\mathbb G_{p,g}(\R)$ is the product of two     
pseudo-unitary groups $SU(58,40)\times SU(44,54)$.
\end{remark}

\subsection{Proofs of Propositions \ref{vcd} and \ref{prop h2rhoiszero}}
\begin{proof}[Proof of Proposition \ref{vcd}]
Let $\Sigma_{g,n}$ denote the compact orientable surface of genus $g$ with $n$ boundary components and 
$M_{g,n}$ the mapping class group of $\Sigma_{g,n}$. 
Then $\Sigma_g$ decomposes into  $g + \left[\frac{g-2}{2}\right]$ pieces with disjoint interiors
among which are $g$ sub-surfaces $\Sigma_{1,1}$,  $\left[\frac{g-2}{2}\right]$ sub-surfaces 
$\Sigma_{0,4}$, and $g-2\left[\frac{g}{2}\right]\in\{0,1\}$ pieces  homeomorphic to $\Sigma_{0,3}$.
 
 If $p\not\in\{2,3,4,6,8,12\}$, $g\geq 2$ and $(p,g)\neq (10,2)$,  then every subgroup of the form $\rho_p(M_{1,1})$ and $\rho_p(M_{0,4})$  
 associated to a sub-surface $\Sigma_{1,1}$ or $\Sigma_{0,4}$ of $\Sigma_g$ 
contains a free non-abelian group $\mathbb F_2$ on two generators (see \cite{FK1,FK2}). In particular,  we find that 
 $\mathbb F_2^{g+ \left[\frac{g-2}{2}\right]}\subset \rho_p(M_g)$. 
 Recall that ${\rm vcd}$ is increasing with respect to the inclusion of groups (see \cite{Brown}, chap VIII, 11, Ex.1, Prop. 2.4). 
Thus 
 \[{\rm  vcd}(\widetilde\rho_p(\widetilde{M_g}))\geq {\rm vcd}(\mathbb F_2^{g+ \left[\frac{g-2}{2}\right]})\geq {\rm vcd}(\Z^{g+ \left[\frac{g-2}{2}\right]})=
 g + \left[\frac{g-2}{2}\right].\]
  Notice that we also have, by the same argument, that 
 ${\rm vcd}(\rho_p(Mg))\geq 
 g + \left[\frac{g-2}{2}\right]$.
Observe that torsion-free nilpotent subgroups of $\widetilde\rho_p(\widetilde{M_g})$ are abelian, because
$\mathbb G_{p,g}(\Z)$ contains no nontrivial unipotents, so that they cannot be used to get better lower bounds. 
 \end{proof}

\begin{remark}
When $p\equiv -1 ({\rm mod} \, 4)$, ${\rm vcd}( \mathbb G_{p,g}(\Z))$ is the dimension of the corresponding non-compact symmetric space, since lattices are cocompact.  
If $\widetilde\rho_p(\widetilde{M_g})$ were of infinite index in $\mathbb G_{p,g}(\Z)$ then
 its top dimensional cohomology  would vanish (see \cite{Brown}, VIII, Prop. 8.1). Therefore $\widetilde\rho_p(\widetilde{M_g})$ has finite index in $\mathbb G_{p,g}(\Z)$ if and only if $ {\rm vcd}(\widetilde\rho_p(\widetilde{M_g}))={\rm vcd}( \mathbb G_{p,g}(\Z))$. 
Compare also with (\cite{Strebel}), where the author proved that passing to an 
infinite index subgroup of a Poincar\'e duality group strictly decreases 
the cohomological dimension. 
\end{remark}

\begin{proof}[Proof of Proposition~\ref{prop h2rhoiszero}]
Since $\rho_p(M_g)$ is of finite index in $\rho_p(\widetilde{M}^u_g)$, the map 
$\rho_p^\ast : H^2(\rho_p(M_g),\R) \rightarrow H^2(M_g,\R)$
factors through $\rho_p^\ast : H^2(\rho_p(\widetilde{M_g}^u),\R) \rightarrow H^2(\widetilde{M_g}^u,\R)$, but the group $\widetilde{M}^u_g$ has no non-split extensions, so this last cohomology group is trivial. 
\end{proof}

\section{Dupont-Guichardet-Wigner quasi-homomorphisms on mapping class groups}\label{dgw}
\subsection{Quasi-homomorphisms on $\widetilde{M_g}$}\label{dgwmorphisms}
 Guichardet-Wigner \cite{GW} and  Dupont 
\cite{D} introduced explicit bounded continuous cocycles $c_{SU(m,n)}$, whose classes generate 
$H^2_b(SU(m,n);\R) \cong \R$ and could be interpreted in terms of the symplectic area of triangles. 
Let $K$ be the maximal compact subgroup 
$S(U(m)\times U(n))$, $A$  the group of  
unitary diagonal matrices with real entries and 
$N$   the group of unitary unipotent matrices in $SU(m,n)$. Corresponding to the Iwasawa decomposition $SU(m,n) = KAN$, we denote by   $x= k(x) a(x) n(x)$
 the Iwasawa decomposition  of the element $x\in SU(m,n)$. The construction due to 
Guichardet and Wigner in (\cite{GW}, Theorem 1) is as follows:
\begin{proposition}\label{GW}
Let $\mathfrak{k}$ be the Lie algebra of the compact group $K$ and $\mathfrak g= \mathfrak k \oplus \mathfrak p$ be the Cartan decomposition of the Lie algebra $\mathfrak g$ of $SU(m,n)$. Consider a smooth function $v: SU(m,n)\to \C^*$ satisfying the 
following conditions: 
\begin{enumerate}
\item the restriction of $v$ to the maximal compact $K$ is a nontrivial 
morphism of $K$ into $U(1)\subset \C^*$; 
\item the restriction of $v$ to $\exp \mathfrak p$ is strictly positive and 
$K$-invariant;  
\item $v(k \cdot \exp p)= v(k) v(\exp p)$, for any $k \in K$ and $p\in \mathfrak p$. 
\end{enumerate}
Then there exists a unique smooth 2-cocycle $c_v:SU(m,n)\times 
SU(m,n)\to \R$ such that 
\[\exp(2\pi \sqrt{-1}c_v(g_1,g_2))= 
{\rm arg }(v(g_1g_2)^{-1}\cdot v(g_1)\cdot v(g_2)), \;{\rm and }\:\: c_v(1,1)=0\] 
Moreover, the class of $c_v$ generates  the Borel cohomology group $H^2(SU(m,n),\R)$. 
\end{proposition}
An example is the function $v_0:K\to U(1)$ 
given by 
$v_0(x)=\det(x_+)$, 
where $x=\left(\begin {array}{cc}
x_+ & 0 \\
0 & x_- \\
\end{array}
\right)\in S(U(m)\times U(n))$ and $x_+$ is the $U(m)$ component of $x$.  
Setting $v_0(\exp p)=1$, and $v_0(k \cdot \exp p) =v_0(k) v_0( \exp p)$, 
extends $v_0$ to a function on all of $SU(m,n)$ with values in $U(1)$  
that satisfies the conditions stated in Proposition \ref{GW}. 
We therefore have the associated continuous bounded cocycle denoted 
$c_{SU(m,n)}$. We will later normalize  the cocycle $c_{SU(m,n)}$ to a cocycle  whose class 
is the generator of the image of 
$H^2(SU(m,n),\Z)$ in $H^2(SU(m,n),\R)$. We also consider the unique  
continuous lift 
$\Phi:\widetilde{SU(m,n)}\to \R$  of $v_0$ to the universal 
covering,  which is  determined by the condition $\Phi(1)=0$.

Let $G$ be a topological group.  The ordinary cohomology group 
$H^2(G,\mathbb{R})$ is usually an extremely large group, for instance 
for non-compact Lie groups its dimension is typically  uncountable (see \cite{Sah}). This is not anymore the case for 
the continuous cohomology of Lie groups and in particular for their  
\emph{bounded} cohomology group $H^2_b(G;\mathbb{R})$.   There is a canonical comparison map 
$H^2_b(G; \mathbb{R}) \rightarrow H^2(G;\mathbb{R})$ whose kernel  is 
described by quasi-homomorphisms:  a map $\varphi:G\to \R$ is a 
{\em quasi-homomorphism} if  $\sup_{a,b\in G}|\partial \varphi(a,b)|< \infty$, 
where 
$\partial \varphi(a,b)=\varphi(ab)-\varphi(a)-\varphi(b)$ is the 
boundary 2-cocycle.  The quasi-homomorphism $\varphi$ is {\em homogeneous}
if $\varphi(a^n)=n\varphi(a)$, for every $a\in G$ and $n\in\Z$. 
Let us denote the vector space of quasi-homomorphisms  by $QH(G)$ and its 
quotient by the subspace generated by the  bounded functions 
 and the group homomorphisms by $\widetilde{QH}(G)$. It is known that  there is an exact sequence: 
\[
 0 \to \widetilde{QH}(G) \rightarrow H^2_b(G; \mathbb{R}) \rightarrow H^2(G;\mathbb{R}).
\]

Bestvina and Fujiwara proved in \cite{BF} that 
$\widetilde{QH}(M_g)$, and hence 
$\widetilde{QH}(\widetilde{M_g})$ has uncountably 
many generators.

Let $g\geq 3$, $p\geq 5$ be  a prime number 
and $SU(m,n)$ be the  non-compact simple factor of $\mathbb G_{p,g}(\R)$ corresponding to the primitive $2p$-th root of unity $\zeta$.   Since the 
universal extension  $\widetilde{M_g}^{\rm u}$ is perfect and 
has no nontrivial extensions, we have an isomorphism 
$\widetilde{QH}(\widetilde{M_g}^{\rm u}) 
\simeq H^2_b(\widetilde{M_g}^{\rm u},\R)$. 
As $\widetilde{M_g}^{\rm u}$ is of finite index 
in $\widetilde{M_g}$, we also have $\widetilde{QH}(\widetilde{M_g}) 
\simeq H^2_b(\widetilde{M_g},\R)$. 
Thus, there exists a  quasi-homomorphism 
$L_{\zeta}:\widetilde{M_g}\to \R$, unique up to a bounded quantity,  verifying  
\[ \partial L_{\zeta}=
\widetilde{\rho}^*_{p,\zeta} (c_{SU(m,n)}). 
\]
Let $\overline{L}_{\zeta}$ denote the unique homogeneous 
quasi-homomorphism in the class of $L_{\zeta}$. 
To give an explicit formula for  the 
quasi-homomorphism $\overline{L}_{\zeta}: \widetilde{M_g} \rightarrow \R$,
we have to introduce  the Dupont-Guichardet-Wigner quasi-homomorphism $\Phi$ on the universal covering $\widetilde{SU(m,n)}$ of 
$SU(m,n)$. 

\begin{definition}\label{DGW} 
A Dupont-Guichardet-Wigner quasi-homomorphism 
$\Phi:\widetilde{SU(m,n)}\to \Q$ 
is a quasi-homomorphism satisfying: 
\[ \Phi(\widetilde{x}\widetilde{y})-\Phi(\widetilde{x})-\Phi(\widetilde{y})=c_{SU(m,n)}(x,y)\]
for all $x,y\in SU(m,n)$ and  their arbitrary lifts $\widetilde{x},
\widetilde{y}\in \widetilde{SU(m,n)}$. 
\end{definition}
The quasi-homomorphism is {\em normalized} if 
\[ \Phi(Tz)= \Phi(z)+1, \; {\rm for } \; z\in \widetilde{SU(m,n)},\]
 where $T$ denotes the generator of  
$\ker(\widetilde{SU(m,n)}\to SU(m,n))$. 
All Dupont-Guichardet-Wigner quasi-homomorphisms are at bounded distance from each other  and 
the unique homogeneous normalized Dupont-Guichardet-Wigner quasi-homomorphism is given by 
$\overline\Phi(z)=\lim_{n\to \infty}\Phi(z^n)/n$.  In fact, it was noticed by Barge and Ghys in \cite[Remarque fondamentale 2]{BG} that there is a unique homogeneous normalized quasi-homomorphism on any central 
extension of a uniformly perfect group, in particular on $\widetilde{SU(m,n)}$. 
The homogeneous quasi-homomorphism associated to a 
continuous quasi-homomorphism is also continuous, by 
the result of Shtern (see \cite{Shtern}, Proposition 1), thus  $\overline{\Phi}$ is continuous.

Barge and Ghys gave a formula for the homogeneous symplectic quasi-homomorphism in (\cite{BG}, Theorem 2.10). 
We will need in the sequel the following extension to the pseudo-unitary case:

\begin{proposition}\label{quasi}
The homogeneous quasi-homomorphism 
$\overline{L}_{\zeta}$ is given by the formula: 
\[ \overline{L}_{\zeta}(x)= \overline\Phi(\widehat{\rho}_{p,\zeta}(x))
\]
where $\widehat{\rho}_{p,\zeta}: \widetilde{M_g}^{\rm u}\to 
\widetilde{SU(m,n)}$ is the  unique lift 
of $\ro_{p,\zeta}(x)$ to $\widetilde{SU(m,n)}$. 
Moreover,  we have: 
\[\overline{L}_{\zeta}(x)\equiv  \frac{1}{2\pi} \left(\sum_{\lambda\in S(\ro_{p,\zeta}(x))}n^+(\lambda){\rm arg}(\lambda)\right)\in \R/\Z\]
where $S(u)$ is the set of eigenvalues of $u$ and 
$n^+(\lambda)$ is the positive multiplicity of $\lambda$ (see section \ref{gw} for details). 
\end{proposition}

\subsection{Non-triviality of the quasi-homomorphisms space}
\begin{proposition}
If $s_{p,g}^{\ast} > r_{p,g}$ then 
$\widetilde{QH}(\ro_{p}(\widetilde{M_g}^{\rm u}))$  cannot be trivial.
\end{proposition}
\begin{proof}
Denote by $i_{p,\zeta}:\ro_{p,\zeta}(\widetilde{M_g}^{\rm u})
\to PU(m,n)$ the obvious inclusion. 

In (\cite{BI},Theorem 1.3) Burger and Iozzi proved that
for any discrete group $\Gamma$, two Zariski dense representations 
$\rho: \Gamma \rightarrow SU(m,n)$, with  $1 \leq m < n$, are non-conjugate if and only if the 
corresponding cohomology classes 
$\rho^\ast(c_{SU(m.n)}) \in H^2_b(\Gamma;\mathbb{R})$ are distinct. 
Moreover, if distinct, then these classes are  
$\mathbb{Q}$-linearly independent. 

Following \cite{F2}, when $\zeta$ runs over the non-compact primitive roots of non-zero signature 
and positive imaginary part  the  bounded classes $i_{p,\zeta}^*(c_{SU(m,n)})\in H^2_b(\ro_p(\widetilde{M_g}^{\rm u}),\R)$  are linearly independent 
over $\Q$. 
If $\widetilde{QH}(\ro_{p}(\widetilde{M_g}^{\rm u}))$ were trivial, 
 then the cohomology classes $i_{p,\zeta}^*(c_{SU(m,n)})\in H^2(\ro_p(\widetilde{M_g}^{\rm u}),\R)$
would also be independent over $\Q$. But these are integral classes, i.e. they lie in the image of 
$H^2(\ro_p(\widetilde{M_g}^{\rm u}),\Z)$, because they 
are pull-backs of integral classes from $H^2(\mathbb G_{p,g}(\R),\R)$. Therefore, they would be linearly 
independent over $\R$. In other words we would produce $s_{p,g}^{\ast}$  linearly independent classes living within the 
vector space ${\rm Hom}(\ker \ro_p, \R)^{\widetilde{M_g}^{\rm u} }$ which 
is of dimension at most $r_{p,g}$. This  contradiction proves the claim. 
\end{proof}

\begin{remark}
If $\widetilde{QH}(\rho_{p}(M_g))$ 
were infinite dimensional then $\rho_p(M_g)$ would not be boundedly generated.
\end{remark}

\subsection{Proof of Proposition \ref{quasi}}
We reduce the problem to the computations made earlier by Barge and Ghys in \cite{BG} in the symplectic 
case. Given an integer $n \geq 1$, let $Sp(2n,\R)$ denote the 
real symplectic group of  $2n\times 2n$ matrices.
There are two natural homomorphisms 
$i: SU(m,n) \hookrightarrow Sp(2(m+n),\R)$ and 
$j: Sp(2n,\R) \hookrightarrow SU(n,n)$, and these lift uniquely to 
continuous group homomorphisms  
$\widetilde{SU}(m,n) \hookrightarrow \widetilde{Sp}(2(m+n),\R)$ and 
$\widetilde{Sp}(2n,\R) \hookrightarrow \widetilde{SU}(n,n)$.   Let us set in this section $\Phi_{SU(m,n)}$ for the homogeneous quasi-homomorphism $\overline{\Phi}$ and $\Phi_{Sp(2n,\R)}$ for its symplectic cousin.
Standard arguments show that:

\begin{proposition}
 \begin{enumerate}
\item The unitary homogeneous quasi-homomorphism $\Phi_{SU(n,n)}$ restricts along the embedding $Sp(2n,\R) \hookrightarrow SU(n,n)$ to the symplectic homogeneous quasi-homomorphism $\Phi_{Sp(2n,\R)}$. 
\item 
 The symplectic homogeneous quasi-homomorphism $\Phi_{Sp(2(m+n),\R)}$ 
restricts along the embedding $SU(m,n) \hookrightarrow Sp(2(m+n),\R)$ 
to $2\Phi_{SU(m,n)}$, if $mn\neq 0$.
\end{enumerate}
\end{proposition}

\begin{remark}
It was already noticed in (\cite{GG}, section 4) 
that the restriction of $\Phi_{Sp(2(m+n),\R)}$ to 
$SU(m+n) \hookrightarrow Sp(2(m+n),\R)$ is trivial, as this 
subgroup is simply connected. The fact that the restriction 
of the Maslov class on $SU(m,n)$ is nontrivial was also 
stated in (\cite{GG}, Corollary 4.4).  
\end{remark} 
 
Then from (\cite{BG}, Theorem 2.10) we deduce the following: 

\begin{proposition}\label{comput}
The  homogeneous Dupont-Guichardet-Wigner quasi-homomorphism 
$\overline{\Phi}: \widetilde{SU(m,n)}\to \R$
is the unique  continuous lift of the 
map $\overline{\phi}:SU(m,n)\to \R/\Z$ sending $1$ to $0$,  defined when $g$ is semi-simple
by the formula: 
\[\overline{\phi}(g) = \frac{1}{2\pi} \left(\sum_{\lambda\in S(g)}n^+(\lambda){\rm arg}(\lambda)\right)\in \R/\Z\]
where $S(g)$ is the set of eigenvalues of $g$ and 
$n^+(\lambda)$ their positive multiplicity. 
\end{proposition}

We postpone the discussion and the definition of positive multiplicity to section \ref{positive}. 

\begin{proof}[End of the proof of Proposition \ref{quasi}]
Proposition \ref{comput} shows that  $\overline{\Phi}$ 
is uniquely determined as a continuous lift of $\overline{\phi}$ and the formula follows because 
$\ro_{p,\zeta}(\widetilde{M_g})\subset SU(m,n)$ consists only of semi-simple elements. 
\end{proof}

 The independence of $\overline{L}_{\zeta}, \overline{\Phi}$ on the chosen 
bounded cocycle $c_{SU(m,n)}$ is a consequence of the fact that $SU(m,n)$ is 
uniformly perfect.

Although the fact that all simple Lie groups are uniformly perfect seems to be 
folklore, we did not find it explicitly in the literature.  
For all semi-simple Lie groups whose maximal compact is 
semi-simple any element is the product of 2 commutators (see 
\cite{Dj}). However this does not apply precisely to $SU(m,n)$.  
One also knows that there are elements which are not commutators 
(from \cite{Th}). An explicit bound for the number of reflections 
needed to write any element in $U(m,n)$ as a product was 
given in \cite{DM} and the number of commutators could be deduced from it. 
Using a similar reasoning one shows that: 

\begin{proposition}
The group $SU(m,n)$ is uniformly perfect, more precisely: any element is a product of at most $14 (m+n)$ commutators.
\end{proposition}

\subsection{Useful properties of Dupont-Guichardet-Wigner cocycles}\label{gw}

Notice that the reduction mod $\Z$ of $\overline{\Phi}$ 
descends to a  map 
$\overline{\phi}:SU(m,n)\to \R/\Z$, given by 
$\overline{\phi}(x)=\overline{\Phi}(\widetilde{x})$, where 
$\widetilde{x}$ is an arbitrary lift of $x$.  The quasi-homomorphism is easy to compute on lifts of  Borel subgroups of $SU(m,n)$ such as $AN$. Recall that all  Borel subgroups of $SU(m,n)$ are conjugate. The subgroup $AN$ is simply connected, and contains the identity matrix, therefore its preimage $\widetilde{AN}$ is a disjoint union of (simply) connected components, each one homeomorphic to $AN$ and canonically indexed by an element of $\mathbb{Z}=\ker (\widetilde{SU}(m,n) \rightarrow SU(m,n))$.

\begin{lemma}\label{lem valueAN}
 The quasi-homorphism  $\overline{\Phi}$ is locally constant on $\widetilde{AN}$. More precisely,   $\overline{\Phi}$ takes the value $d$ on the sheet of $\widetilde{AN}$ indexed by $d$.  Consequently, if $B$ is an arbitrary Borel subgroup of $SU(m,n)$ and $\widetilde{B} $ denotes its preimage in $\widetilde{SU}(m,n)$, then $\overline{\Phi}$ takes integer values on $\widetilde{B}$.
\end{lemma}
\begin{proof}
By construction the function $v_0$ is constant 
with value $1$ on $AN$, therefore its continuous lift 
$\Phi$ takes integral values on the $\widetilde{AN}$, and as it 
is continuous, these values are given by the integer indexing the 
connected component. Moreover,  if $g \in \widetilde{AN}$ belongs 
to the component indexed say by $d$, then for any $n \in \Z$ the element  
$g^n$ belongs to the component indexed by $nd$. Therefore we have: 
\[\overline{\Phi}(g) =\lim_{n\to \infty}\frac 1n\Phi(g^n)=\lim_{n\to \infty}\frac 1n dn = d.\]

If $B$ is an arbitrary Borel subgroup, then there is an element $g \in SU(m,n)$ such that $gBg^{-1} \subset AN$. As a consequence, if we denote by $\widetilde{g}$ a preimage of $g$ in $\widetilde{SU}(m,n)$, conjugation by $\widetilde{g}$ embeds  $\widetilde{B}$ into $\widetilde{AN}$. As $\overline{\Phi}$ is invariant under conjugation the result follows. 
\end{proof}

\begin{proposition}
The homogeneous normalized quasi-homomorphism on 
$\widetilde{SU}(m,n)$ is the unique continuous 
normalized lift of the map 
$\overline{\phi}\circ e:SU(m,n)\to \R/\Z$ 
where  $g= e(g) h(g) u(g)$
is the Jordan decomposition of $g\in SU(m,n)$.
Recall that $e(g)$ is the elliptic part, $h(g)$ the hyperbolic part and 
$u(g)$ the unipotent part of $g$.  
\end{proposition}
\begin{proof} Let $g$ be an arbitrary element in $SU(m,n)$ and $\widetilde{g} \in \widetilde{SU}(m,n)$ one of its lifts. Choose also a lift $\widetilde{e}(g)$ of $e(g)$.  Since $e(g)$ commutes with $g$ we have that $\overline{\Phi}(\widetilde{e}(g)^{-1}\widetilde{g}) = \overline{\Phi}(\widetilde{e}(g)^{-1}) + \overline{\Phi}(\widetilde{g})$. By construction, $\widetilde{e}(g)^{-1}\widetilde{g} = h(g)u(g)$ and since  $h(g)$ is conjugate to some element in $A$ and $u(g)$ to some element in $N$, $h(g)u(g)$ belongs to some Borel subgroup of $SU(m,n)$. By Lemma~\ref{lem valueAN} this implies that $\overline{\Phi}(\widetilde{e}(g)^{-1}\widetilde{g}) \in \mathbb{Z}$ or equivalently:

\begin{eqnarray*}
\overline{\Phi}(\widetilde{g}) & =&  - \overline{\Phi}(\widetilde{e}(g)^{-1}) \text{ mod } \mathbb{Z} \\
 & =&  \overline{\Phi}(\widetilde{e}(g)) \text{ mod } \mathbb{Z} \\
 & =& \overline{\phi}(e(g)) \text{ by definition of } \overline{\phi}.
\end{eqnarray*}
The second equality comes from the fact that, as $\overline{\Phi}$ is homogeneous and normalized, for any $h \in \widetilde{SU}(m,n)$, $\overline{\Phi}(h^{-1}) =  -\overline{\Phi}(h)$.
\end{proof}

\subsection{Positive eigenvalues of pseudo-unitary operators}\label{positive}

Consider a pseudo-unitary operator $g\in SU(m,n)$.
Let $H:V\times V\to \C$ be the indefinite Hermitian form 
defining the group $SU(m,n)$, where $\dim_{\C} V=m+n$.  
We will assume henceforth that $1\leq m \leq n$.

The spectrum $S(g)$   
of $g$ is symmetric with respect to the 
unit circle, namely if $\lambda\in S(g)$ then 
$\overline{\lambda}^{-1}\in S(g)$ (see \cite{GL}, ch. 10, section 5). 
For a given $\lambda\in S(g)$  we consider the root space
$V_{\lambda}(g)=\ker (g-\lambda I)^{m+n}\subset V$.  
We have then $V=\oplus_{\lambda\in S(g)}V_{\lambda}(g)$. 
Moreover, each $V_{\lambda}(g)$ splits as 
$V_{\lambda}(g)=\oplus_{i}V_{\lambda,i}(g)$, 
where each subspace $V_{\lambda,i}(g)$ 
corresponds to a Jordan block  with diagonal $\lambda$ 
in the Jordan decomposition of $g$. The number of such 
subspaces $V_{\lambda,i}(g)$ (i.e. Jordan blocks) is 
the geometric multiplicity of $\lambda$, namely $\dim \ker(g-\lambda I)$. 
The collection of dimensions $\dim V_{\lambda,i}$ is 
the collection of partial multiplicities of $\lambda$. 
Furthermore the collection of partial multiplicities of 
$\lambda\in S(g)$ agrees with the one for 
$\overline{\lambda}^{-1}$.

We will use the canonical form of pseudo-unitary operators 
from (\cite{GLR}, Theorem 5.15.1). We will only need a weaker form and 
state it in a simplest form, though the statement in \cite{GLR} is 
more precise: 

\begin{proposition}\label{can}
Let $g\in SU(m,n)$ have the set of Jordan blocks 
$J_1,J_2,\ldots, J_{a+2b}$ (where $a+2b\leq m+n$) and corresponding 
eigenvalues  $\lambda_1, \lambda_2,\ldots,\lambda_{a+2b}$, 
not necessarily distinct. We suppose that 
that $|\lambda_1|=|\lambda_2|=\cdots=|\lambda_a|=1$, 
$|\lambda_{a+2i-1}| >1$ and 
$\lambda_{a+2i-1}=\overline{\lambda}^{-1}_{a+2i}$, for $1\leq i\leq b$.  
Then there exists a non-singular matrix $C$ such that 
the following two conditions hold simultaneously: 
\[ C^{-1}gC= \bigoplus_{i=1}^{m^+(g)}\lambda_{j_i}K_{j_i} \bigoplus_{i=1}^{m^-(g)}\lambda_{s_i}K_{s_i} \bigoplus_{1\leq i\leq b}\left(\begin{array}{cc}
\lambda_{a+2i-1}K_{a+2i-1} & 0 \\
0 & \overline{\lambda}_{a+2i-1}^{-1}K_{a+2i} \\
\end{array}\right),
\]
\[ C^*HC= \bigoplus_{i=1}^{m^+(g)}P_{j_i} 
\bigoplus_{i=1}^{m^-(g)} - P_{s_i} \bigoplus_{1\leq i\leq b }\left(\begin{array}{cc}
0 & P_{a+2i-1} \\
P_{a+2i} & 0 \\
\end{array}\right),
\]
where 
\begin{enumerate}
\item The blocks $K_j$ are unipotent upper triangular matrices (also called Toeplitz blocks), for all $j\leq a+2b$;
\item Each matrix $P_j$ is a permutation matrix of the form 
$\left(\begin{array}{cccccc}
0 & 0 & 0 & \cdots & 0 & 1\\ 
0 & 0 & 0 & \cdots & 1 & 0 \\
\vdots & \vdots & \vdots & \vdots & \vdots & \vdots\\
0 & 1 & 0 & \cdots & 0 & 0 \\
1 & 0 & 0 & \cdots & 0 & 0 \\
\end{array}
\right)
$ having the size of the Jordan block $J_j$, for all $j\leq a+2b$; 
\item The two sets $\{j_1,j_2,\ldots, j_{m_+(g)}\}$ and  
$\{s_1,s_2,\ldots, s_{m_-(g)}\}$ form a partition of $\{1,2,\ldots,a\}$, 
so that $m_+(g)+m_-(g)=a$. The sign characteristic 
$\varepsilon_i\in\{\pm 1\}$, for $1\leq i\leq a$ is given by 
$\varepsilon_i=1$ iff $i\in \{j_1,j_2,\ldots, j_{m_+(g)}\}$.  
 \item The canonical form is unique, up to a permutation of equal Toeplitz 
blocks respecting the sign characteristic. 
\end{enumerate}
\end{proposition}

\vspace{0.2cm}
\noindent 
When $g$ is semi-simple the canonical form is simpler, as follows: 
 
\begin{corollary}
Let $g\in SU(m,n)$ be a semi-simple element with 
eigenvalues  $\lambda_i$, $1\leq i\leq m+n$. Let us denote by $\lambda_{\alpha}, \overline{\lambda}_{\alpha}^{-1}$, with $\alpha\in N(g)\subset 
\{1,2,\ldots,m+n\}$ those eigenvalues of modulus different from $1$, where 
$|\lambda_{\alpha}| >1$. 
 Then there exists 
a non-singular matrix $C$ such that the following two conditions 
hold simultaneously: 
\[ C^{-1}gC= \oplus_{i=1}^{m^+(g)}(\lambda_{j_i}) \oplus 
\oplus_{i=1}^{m^-(g)}(\lambda_{s_i}) \oplus \oplus_{\alpha\in N(g)}\left(\begin{array}{cc}
\lambda_{\alpha} & 0 \\
0 & \overline{\lambda}_{\alpha}^{-1} \\
\end{array}\right),
\]
\[ C^*HC= \oplus_{i=1}^{m^+(g)}(+1) \oplus 
\oplus_{i=1}^{m^-(g)}(-1) \oplus \oplus_{\alpha\in N(g)}\left(\begin{array}{cc}
0 & 1 \\
1 & 0 \\
\end{array}\right).
\]
Here the sets of indices $\{j_1,j_2,\ldots, j_{m_+(g)}\}$,  
 $\{s_1,s_2,\ldots, s_{m_+(g)}\}$ and $N(g)$ form a partition of 
$\{1,2,\ldots,m+n\}$. The canonical form is unique up to a permutation 
preserving  the eigenvalues and the sign characteristic. 
\end{corollary}
\begin{proof}
This result seems to have been  stated explicitly  first by Krein (see \cite{Kr}) 
for the symplectic group and by Yakubovich in the present setting 
(see \cite{Ya}, p.124).  
\end{proof}

\begin{definition}
Let $g$ be a semi-simple element of $SU(m,n)$. 
The  eigenvalues $\lambda_{i}$ of $g$, 
for $i\in \{j_1,j_2,\ldots, j_{m_+(g)}\}$, i.e. those for which 
$\varepsilon_i=+1$, 
will be called {\em positive} (after Gelfand and Lidskii, Krein and Yakubovich) and their  {\em  positive multiplicity} $n_i^+$ is 
the multiplicity among positive eigenvalues.
By convention, the eigenvalues $\lambda_{\alpha}$ with 
$|\lambda_{\alpha}| >1$ are said to be {\em positive} and their 
positive multiplicity coincide with the usual multiplicity.
The remaining  eigenvalues will be called {\em negative} eigenvalues of $g$.
We will also denote by $n^+(\lambda)$ the positive multiplicity 
of the eigenvalue $\lambda$ (which is $0$ for negative ones) 
of the semi-simple $g$.  
\end{definition}

The positivity seems more subtle when $g$ is not semi-simple.
In fact the signature of each block $\varepsilon_j P_j$  
equals $0$ when its dimension $n_j$ is even and 
$\varepsilon_j$, when its dimension $n_j$ is odd, respectively. 
Further, the signature of  $\left(\begin{array}{cc}
0 & P_{a+2i-1} \\
P_{a+2i} & 0 \\
\end{array}\right)$ 
is always $0$. Thus, every eigenvalue involved in a Jordan block 
is positive with a positive multiplicity equal to approximatively 
half of its partial multiplicity. 

\begin{lemma}
Let  $g\in SU(m,n)$. Then in a suitable basis of $V$ we can write 
simultaneously:
\[ e(g)= \bigoplus_{i=1}^{a}{\rm diag }(\lambda_{i})\bigoplus_{1\leq i\leq b}\left(\begin{array}{cc}
{\rm diag }(\frac{\lambda_{a+2i-1}}{|\lambda_{a+2i-1}|}) & 0 \\
0 &  {\rm diag }(\frac{\lambda_{a+2i-1}}{|\lambda_{a+2i-1}|})  \\
\end{array}\right),
\]
\[ H= \bigoplus_{i=1}^{a}\varepsilon_iX_i
\bigoplus_{1\leq i\leq b }\left(\begin{array}{cc}
I & 0 \\
0 & -I \\
\end{array}\right),
\] 
where  ${\rm diag }(\lambda_{i})$ is a diagonal matrix of the size equal 
to the partial multiplicity $n_i$ of $\lambda_i$ and 
$X_i$ is the diagonal matrix of the same size 
with entries $\pm 1$ of signature $\frac12(1-(-1)^{n_i})$.   
\end{lemma}
\begin{proof}
Clear.
\end{proof}

Furthermore, the elliptic element $e(g)$ is conjugate to some 
element $\left(\begin{array}{cc}
e(g)_+ & 0 \\
0 & e(g)_-\\
\end{array}\right)$ of $S(U(m)\times U(n))$, where 
$e(g)_+\in U(m)$ corresponds to a maximal invariant positive 
subspace of $V$ for the Hermitian form $H$. 
The previous lemma gives an explicit formula for 
$e(g)_+$ in the form: 
\[ e(g)_+= \bigoplus_{i=1}^{a}{\rm diag}_+(\lambda_{i})
\bigoplus_{1\leq i\leq b}
{\rm diag}(\frac{\lambda_{a+2i-1}}{|\lambda_{a+2i-1}|}),
\]
where  ${\rm diag}_+(\lambda_{i})$ is a diagonal matrix of the size equal 
to its partial positive multiplicity, defined as:  
$n_i^+=\left\{\begin{array}{ll} 
\frac{n_i}{2}, & {\rm  even } \; n_i\\
\frac{n_i+\varepsilon_i}{2}, & {\rm odd } \; n_i\\
\end{array}\right.
$.

An immediate consequence is that 
\[ \det (e(g)_+)= \exp\left(2\pi\sqrt{-1}\left(\sum_{i=1}^an_i^+{\rm arg}(\lambda_i)
+\sum_{i=1}^bn_{a+2i-1}{\rm arg}(\lambda_{a+2i-1})\right)\right).\]
When $g$ is already semi-simple this formula simplifies to 
\[ \det (e(g)_+)= \exp\left(2\pi\sqrt{-1}\left(\sum_{\lambda\in S(g)}n^+(\lambda){\rm arg}(\lambda)\right)\right).\]

We formulate the result obtained so far in the following: 

\begin{lemma}
For $g\in SU(m,n)$ we have 
\[\overline{\phi}(g) = \frac{1}{2\pi} 
\left(\sum_{i=1}^an_i^+{\rm arg}(\lambda_i)
+\sum_{i=1}^bn_{a+2i-1}{\rm arg}(\lambda_{a+2i-1})\right)
\in \R/\Z.\]
\end{lemma}

\section{Evaluation of quasi-homomorphisms}
\subsection{Arithmetic properties of  dimensions of conformal blocks}\label{compute}
The aim of this section is to provide ground for the 
explicit computations of values of quasi-homomorphisms in the next section. 
Our results here are far from being complete and might only be seen 
as quantitative evidence in the favor of various non-degeneracy 
conditions of arithmetic nature. 

\subsubsection{Dimensions}
 
The first step is an apparently unnoticed congruence satisfied 
by the dimensions $N(g,p)$ of the space of conformal blocks arising 
in the TQFT $\mathcal V_p$. Before  proceeding we need to introduce some 
notation.

We denote by $\theta(p)$ the order of the root of unity 
$\zeta_{2p}^{-12-p(p+1)}$, where $\zeta_{2p}$ is a primitive $2p$-th 
root of unity. Specifically, we have: 

\begin{lemma}
\begin{enumerate}
\item If $p$ is odd we have: 
\[ \theta(p)=\left\{\begin{array}{ll}
p, & {\rm if } \; {\rm g.c.d.}(p,6)=1\\
\frac{p}{3}, &  {\rm if } \; p\equiv 0 ({\rm mod }\; 3) \\
\end{array}\right.\]
\item Assume $p$ is even. 
\begin{enumerate}
\item If $p=12s$, $s\in \Z$ 
\[ \theta(p)=\left\{\begin{array}{ll}
2s, & {\rm if } \;  s\equiv 0 ({\rm mod }\; 2)   \\
s, &  {\rm if } \; s\equiv 1 ({\rm mod }\; 2) \\
\end{array}\right.\]
\item If $p=4s$, $s\in \Z$, ${\rm g.c.d.}(s,3)=1$ 
\[ \theta(p)=\left\{\begin{array}{ll}
2s, & {\rm if } \;  s\equiv 0 ({\rm mod }\; 2)   \\
s, &  {\rm if } \; s\equiv 1 ({\rm mod }\; 2) \\
\end{array}\right.\]
\item If $p=6s$, $s\in \Z$, ${\rm g.c.d.}(s,2)=1$ 
then $\theta(p)=2s$. 
\item  If $p=2s$, $s\in \Z$, ${\rm g.c.d.}(s,6)=1$ 
then $\theta(p)=2s$. 
\end{enumerate}
\end{enumerate}
\end{lemma}
\begin{proof}
Direct calculation. 
\end{proof}

\begin{proposition}\label{cong}
If $g\geq 3$ then 
\[ N(g,p)\equiv 0 ({\rm mod}\; \theta(p)).\]
If $g=2$ then 
\[ 10 N(g,p)\equiv 0 ({\rm mod}\; \theta(p)).\]
\end{proposition}
\begin{proof}
The universal central extension 
$\widetilde{M_g}^u$ is a subgroup of the central extension 
$\widetilde{M_g}(12)$ arising in the TQFT representation, 
which has Euler class 12 (see \cite{MR}). It was already noticed 
in \cite{DW,FP} that the image $\ro_p(\widetilde{M_g}^u)$ 
in the unitary group $U(N(g,p))$ is actually contained in 
the subgroup $SU(N(g,p))$ for $g\geq 3$. This is a consequence of the fact 
that $\widetilde{M_g}^u$ is perfect.
The action of the central element of $\widetilde{M_g}^u$ 
is by means of the scalar matrix $\zeta_{2p}^{-12-p(p+1)}$ (see 
e.g. \cite{MR}). This matrix has therefore unit determinant and 
hence the first congruence follows. In the case $g=2$ we have to use 
the fact that $H_1(\mathcal M_2)=\Z/10\Z$ and follow the same 
lines.     
\end{proof}

We have also for small values of the genus $g$ the following computations dues 
to Zagier (\cite{Za}):
\[ N(g, 2k)=\left\{\begin{array}{ll}
\frac{1}{6}(k^3-k), & {\rm if } \; g=2\\
\frac{1}{180}(k^2(k^2-1)(k^2+11), &  {\rm if } \; g=3 \\
\frac{1}{7560}(k^3(k^2-1)(2k^4+23k^2+191), &  {\rm if } \; g=4 \\
\end{array}\right.\]
and from \cite{BHMV}: 

\[ N(g,p) = \frac{1}{2^g} N(g, 2p), {\rm if } \; p \; {\rm is \; odd}.\]

Notice that, with the notations from \cite{Za} 
we have $N(g,p)=\mathcal D(g,k)$, when $p=2k$ and 
$N(g,p)= \frac{1}{2^g} \mathcal D(g, p)$ if $p$ is odd. 
As an immediate corollary we obtain the following: 

\begin{lemma}\label{odd}
\begin{enumerate}
\item If $g=3$ and $p=4n+2$ or $p=8n\pm 3$ then 
$N(3,p)$ is odd. 
\item If $p=5$ then $N(g,5)$ is odd iff the  
genus $g\not\equiv 1({\rm mod }\; 3)$.  
\end{enumerate}
\end{lemma}
\begin{proof}
Using the Verlinde formula (usually for even $p$) and the previous 
relation we find that 
the dimension $N(g,5)$ is given by: 
\[ N(g,5)=\left(\frac{5+\sqrt{5}}{2}\right)^{g-1}+ 
\left(\frac{5-\sqrt{5}}{2}\right)^{g-1}.\]
Thus $N(g,5)$ is determined by the following recurrence with 
the given initial conditions:  
\[N(g+1,5)=5N(g,5)-5N(g-1,5), \; N(1,5)=2, N(2,5)=5.\]
The mod 2 congruence follows by induction on $g$. 
\end{proof}

\begin{corollary}
The signature is non-zero (as needed in \cite{BI}) 
when $N(g,p)$ is odd, and thus 
for infinitely many values of $g,p$ as in Lemma \ref{odd}. 
\end{corollary}

\subsubsection{Signatures}
The Verlinde formula for the dimensions $N(g,p)$ admits refinements 
for the case of the signatures $\sigma(g,\zeta_{2p})$ of the 
Hermitian forms $H_{\zeta}$ in genus $g$. Here the 
root of unity $\zeta_{2p}$ is a primitive $2p$-th root of unity.
 More details will 
appear in a forthcoming paper \cite{CF} devoted to this subject. 
The aim of this section is to gather evidence to back-up the   following: 

\begin{conjecture}\label{arithconj}
Let us consider $\zeta$ a primitive $2p$-th root of unity, 
for prime $p\geq 5$ such that neither $\zeta$ nor $\overline{\zeta}$ are equal to 
$A_p$, for odd $p$  and $\pm A_p$, for even $p$,
respectively. 
Then for all $g$ in some arithmetic progression 
$\sigma(g,\zeta)\not\equiv 0 ({\rm mod}\; p)$.  
\end{conjecture}

We have the following general behavior: 
\begin{proposition}[\cite{CF}]
For each $p$ we have:
\[ \sigma(g, p, \zeta)= \sum_{i=1}^{\left[\frac{p-1}{2}\right]}
\lambda_i(\zeta)^{g-1},\]
where $\lambda_i(\zeta)$ run over the set of roots of some 
polynomials $P_{\zeta}$ with integer coefficients.  
\end{proposition}
\begin{remark}
Observe that $N(1,p)=\left[\frac{p-1}{2}\right]$, which corresponds to the 
fact that $\sigma(g, p, \zeta)=N(1,p)$ for any $\zeta$, because 
the genus one Hermitian form $H_{\zeta}$ is always positive, as 
the image of the quantum representations is always finite (see e.g. 
\cite{Gilmer}). 
\end{remark}

In this section we will denote $\zeta_{2p}=\exp(\frac{\pi i}{p})$ 
the principal primitive root of unity. The other primitive roots 
of unity are of the form $\zeta_{2p}^k$, with odd $k$. Moreover, it is 
enough to restrict to the case when $k\in\{1,3,5,\ldots p-1\}$. 
Recall that $P_{\zeta_{2p}^{\frac{p-1}{2}}}$, for 
$p\equiv -1({\rm mod }\; 4)$ and $P_{\zeta_{2p}^{\frac{p+1}{2}}}$ for 
$p\equiv 1 ({\rm mod }\; 4)$, respectively are the polynomials 
associated to the unitary TQFTs, thereby computing the dimensions 
of the space of conformal blocks according to the Verlinde formula. 
With the help of a computer program ran by 
F. Costantino one finds that: 
\begin{example}
\begin{enumerate}
\item Let  $p=5$. 
\begin{enumerate}
\item  We have: 
\[ P_{\zeta_{10}}= x^2 - 3x + 3 \]
and the first terms of the sequence $\sigma(g,5,\zeta_{10})$, $g\geq 1$ 
are 
\[  2, 3, 3, 0, -9, -27, -54, -81, -81, 0, 243.\]
\item Further 
\[ P_{\zeta_{10}^3}= x^2 - 5x + 5 \]
and the first terms of the sequence $\sigma(g,5,\zeta_{10})$, $g\geq 1$ 
are the dimensions $N(g,5)$:
\[ 2, 5, 15, 50, 175, 625, 2250, 8125, 29375, 106250, 384375.\]
\end{enumerate}
\item Let $p=7$.  
\begin{enumerate}
\item We have 
\[ P_{\zeta_{14}}= x^3 - 8x^2 + 23 x - 23 \]
and the first terms of the sequence $\sigma(g,7,\zeta_{14})$, $g\geq 1$ 
are 
\[ 3, 8, 18, 29, 2, -237, -1275, -4703, -13750, -31156, -41167\]
\item Also 
\[ P_{\zeta_{14}^3}= x^3 - 14x^2 + 49 x - 49 \]
and the first terms of the sequence $\sigma(g,7,\zeta_{14}^3)$, $g\geq 1$ 
are given by the dimension $N(g,7)$:
\[3, 14, 98, 833, 7546, 69629, 645869, 6000099, 55765626, 518361494,
   4818550093.\]
\item Eventually we have:  
\[ P_{\zeta_{14}^5}= x^3 - 6x^2 + 23 x - 23 \]
and the first terms of the sequence $\sigma(g,7,\zeta_{14}^5)$, $g\geq 1$ 
are: 
\[3, 6, -10, -129, -406, 301, 8177, 32801, 15658, -472404, -2440135.\]
\end{enumerate}
\item Let $p=9$.  
\begin{enumerate}
\item We have 
\[ P_{\zeta_{18}}= x^4 - 16x^3 + 97 x^2 - 257 x + 257\]
and the first terms of the sequence $\sigma(g,9,\zeta_{18})$, $g\geq 1$ 
are 
\[4, 16, 62, 211,
  446, -1509, -29113, -259040, -1823114, -11137172, -60443933.\]
\item Further 
\[ P_{\zeta_{18}^5}= x^4 - 30x^3 + 243 x^2 - 729 x + 729\]
and the first terms of the sequence $\sigma(g,9,\zeta_{18}^5)$, $g\geq 1$ 
are the dimensions $N(g,9)$:
\[
\begin{array}{l}
4, 30, 414, 7317, 137862, 2637765, 50664771, 974133540, 18734896134,
   360344121174, \\  6930952607259. 
\end{array}
\]

\item Eventually  
\[ P_{\zeta_{18}^7}= x^4 - 10x^3 + 101 x^2 - 257 x + 257\]
and the first terms of the sequence $\sigma(g,9,\zeta_{18}^7)$, $g\geq 1$ 
are 
\[4, 10, -102, -1259, -746, 90915,
   687147, -2179104, -67636010, -303038972, 3064220783.\]
\end{enumerate}
\end{enumerate}
\end{example}

\begin{remark}
We have $P_{\zeta}=P_{\overline{\zeta}}$. 
Moreover, for even $p$ we also have 
$P_{\zeta}=P_{-\zeta}$. 
\end{remark}

\begin{proposition}\label{579}
Conjecture \ref{arithconj} is true for $p\in\{5,7,9\}$. 
\end{proposition}
\begin{proof}
We obtain from above that the sequence  
$\sigma(g,\zeta_{10})({\rm mod}\; 5)$, $g\geq 1$  
is periodic with period 24 and its terms read: 
\[ 2,3,3,0,1,3,1,4,4,0,3,4,3,2,2,0,4,2,4,1,1,1,0,2,1,2,3,...\]
Therefore 
$\sigma(g,\zeta_{10})\equiv 0({\rm mod}\; 5)$ if and only if 
$g ({\rm mod }\; 24)\in \{4, 10, 16, 22\}$.

Furthermore, for $p=7$ the sequence  
$\sigma(g,\zeta_{14})({\rm mod}\; 7)$, $g\geq 1$  
is periodic with period 12 and its first terms read: 
\[ 3,1,4,1,2,1,-1,1,5,1,0,1,3,1,4,...\]
Thus $\sigma(g,\zeta_{14})\equiv 0({\rm mod}\; 7)$ if and only if 
$g\equiv 11 ({\rm mod }\; 12)$.

The sequence $\sigma(g,\zeta_{14}^3)({\rm mod}\; 7)$, $g\geq 1$ is 
eventually periodic. One can check that 
$\sigma(g+36,\zeta_{14}^3)\equiv \sigma(g,\zeta_{14}^3)({\rm mod}\; 7)$
for $g\geq 55$.

A more conceptual proof is as follows. 
It suffices to show that $P_{\zeta}(0)$ is invertible $({\rm mod }\; p)$.   
The vector  $v_g=\left(\sigma(h,\zeta)_{h\in\{g,g+1,\ldots, 
g+\left[\frac{p-1}{2}\right]-1\}}\right)$ is obtained from $v_1$ by means of the formula 
\[ v_g= M_{\zeta}^g v_1\]
where $M_{\zeta}$ is the companion matrix associated to $P_{\zeta}$.
Therefore $\det M_{\zeta}= P_{\zeta}(0)$. If the determinant is invertible   
mod $p$ then the sequence of vectors $M_{\zeta}^g v_1$ cannot contain 
the null vector mod $p$. But this sequence is eventually periodic. 
Therefore for $g$ in some arithmetic progression $\sigma(g,\zeta)$ 
is nontrivial mod $p$.  Using the 
explicit values of $P_{\zeta}$ one settles immediately 
the cases $p\in\{5,7,9\}$.  
\end{proof}

\subsection{Proof of  Theorem \ref{lower}}\label{llower}
Recall from section \ref{dgwmorphisms} that we have a homogeneous quasi-homomorphism 
$\overline{L}_{\zeta}: \widetilde{M_g}\to \R$ associated to a primitive $2p$-th root of unity $\zeta$. 
Consider  the  map
\[\overline{l}_{\zeta}= \overline{L}_{\zeta}|_{\ker \ro_p}:\ker \ro_p\to \R.\]

\begin{lemma}
We have $\overline{l}_{\zeta}\in {\rm Hom}(\ker \ro_p, \R)^{\widetilde{M_g}^{\rm u}}$, namely $\overline{l}_{\zeta}$ is a group homomorphism invariant by the conjugacy action 
of $\widetilde{M_g}^{\rm u}$. 
\end{lemma}
\begin{proof} 
The boundary of $\overline{L}_{\zeta}$ is $\ro_p^*(c_{SU(m,n)})$ 
which obviously vanishes on $\ker \ro_p$, namely 
\[ \ro_p^*(c_{SU(m,n)})(x,y)=0, \; {\rm if \; either } 
\; x \: {\rm or } \; y \in \ker \ro_p.\]
This implies that $\overline{l}_{\zeta}$ is a homomorphism. 

\vspace{0.2cm}\noindent
Eventually recall that  $\overline{L}_{\zeta}$ is a homogeneous 
quasi-homomorphism and thus it is a class function. This implies that 
$\overline{l}_{\zeta}$ is also a class function. 
\end{proof}

Recall from the proof of Proposition \ref{hom} that there is an isomorphism
\[ \iota: {\rm Hom}(\ker \ro_p, \R)^{\widetilde{M_g}^{\rm u}}\to H^2(\ro_p(\widetilde{M_g}^{\rm u}),\R).\]
We want to show that $\overline{l}_{\zeta}\neq 0$  and consequently 
$\iota(\overline{l}_{\zeta})\in H^2(\ro_p(\widetilde{M_g}^{\rm u}),\R)$ is not vanishing. 
Denote by $h_g^+(\zeta)$ the dimension 
of the maximal positive subspace of the Hermitian form $H_{\zeta}$.

\begin{proposition}\label{central}
Suppose that $h_g^+(\zeta)\not\equiv 0 \;({\rm mod}\; p)$, $p\geq 5$ prime. Then 
$\iota(\overline{l}_{\zeta})\neq 0\in H^2(\ro_p(\widetilde{M_g}^{\rm u});\R)$.
\end{proposition}
\begin{proof}
Let $c$ denote a generator of the center of 
$\widetilde{M_g}^{\rm u}$. 
We know that $\ro_{p,\zeta}(c)=\zeta^{-12}$ (see \cite{MR}), when $p$ is odd. 
The formula of Proposition \ref{quasi} yields 
\[ \overline{L}_{\zeta}(c)\equiv -12 h_g^+(\zeta){\rm arg}(\zeta) \;\;\;({\rm mod}\; 2\pi \Z).\]
Now, if $\overline{L}_{\zeta}(c)\not\equiv 0\in \R/2\pi\Z$, then 
$\overline{L}_{\zeta}(c)\neq 0$. This implies that 
$\overline{L}_{\zeta}(c^n) \neq 0$ for any $n\neq 0$. 
Recall that $c^p\in \ker \ro_{p,\zeta}$. Thus $\overline{l}_{\zeta}(c^p)\neq 0$ so that 
 $\overline{l}_{\zeta}$ is not identically zero, as claimed.  
\end{proof}

\begin{proposition}\label{nonzero}
If $p\in\{5,7,9\}$ then $\overline{l}_{\zeta_{2p}}$ is non-zero for 
infinitely many values of $g$ in some arithmetic progression.   
\end{proposition}
\begin{proof}
According to Proposition \ref{central} it suffices to show that 
$h^+(\zeta_{2p})\not\equiv 0 ({\rm mod }\; p)$. We proved in 
Proposition \ref{cong} that $N(g,p)\equiv 0 ({\rm mod }\; p)$, so that 
this condition is equivalent to proving that 
$\sigma(g,\zeta_{2p})\not\equiv 0 ({\rm mod }\; p)$.
But this last statement is part of Proposition \ref{579}. 
\end{proof}

\vspace{0.2cm}\noindent
{\em End of the proof of Theorem \ref{lower}}. 
From Proposition \ref{nonzero} and the proof of Proposition \ref{hom} 
we obtain that  $\overline{l}_{\zeta_{2p}}$ is non-zero and hence 
a nontrivial class in $H^2(\ro_p(\widetilde{M_g}^{\rm u}),\R)$ for 
infinitely many values of $g$ and $p\in\{5,7,9\}$.

\vspace{0.2cm}

\begin{remark}
The same method provides examples when $\overline{L}_{\zeta}(T_{\gamma}^p)\neq 0$, and hence   
slightly better lower bounds  for the rank of 
$H^2(\ro_p(\widetilde{M_g}^{\rm u}),\R)$.  
\end{remark}

\begin{remark}
If we were able to show that there is at least one nontrivial 
quasi-homomorphism on $\rho_{p,\zeta}(M_g)$ then it would follow 
that this group cannot be an irreducible higher rank lattice in a semi-simple 
Lie group, according to  the result of Burger and Monod from \cite{BM}.  
\end{remark}

{
\small      
      
\bibliographystyle{plain}

}

\end{document}